\documentclass[a4paper,12pt]{amsart}
\usepackage{amssymb}
\usepackage{paralist}
\usepackage{curves}
\usepackage{lamsarrow}		
\usepackage{pb-diagram,pb-lams}	
\usepackage{pdfsync}		
\newtheorem{thm}{Theorem}
\newtheorem{lem}[thm]{Lemma}
\newtheorem{defn}[thm]{Definition}
\newtheorem{prop}[thm]{Proposition}

\newtheorem{cor}[thm]{Corollary}
\newtheorem{conj}[thm]{Conjecture}
\newtheorem{expl}[thm]{Example}

\newtheorem{rem}[thm]{Remark}

\def\homeo{\approx}
\def\Map{\operatorname{Map}}
\def\Hom{\operatorname{Hom}}
\def\Min{\operatorname{Min}}
\def\Max{\operatorname{Max}}
\def\res{\operatorname{res}}
\def\Int{\operatorname{Int}}
\def\Clo{\operatorname{Cl}}
\def\Supp{\operatorname{Supp}}

\def\subdiv{\operatorname{Sd}}
\def\tentdiv{\operatorname{Td}}

\def\subdivision{\Category{SubDiv}}
\def\ast{\hbox{\footnotesize$*$}}
\makeatletter
\def\hitem{\addtocounter{enumi}{1}\hspace{1em}\text{\rm(\number\value{enumi})\,~}}
\def\hiitem{\addtocounter{enumii}{1}\hspace{1em}\text{\alph{enumii})\,~}}

\makeatother
\def\Category#1{{\sf{#1}}}
\def\Object#1{\mathcal{O}(\text{\small \(#1\)})}
\def\Morphism#1{\mathcal{M}(\text{\small \(#1\)})}
\def\mathbold#1{\mbox{\boldmath $#1$}}
\def\diffeological{\Category{Diffeology}}
\def\differentiable{\Category{Differentiable}}

\def\topological{\Category{Topology}}
\def\domain{\Category{Domain}}
\def\open{\Category{Open}}
\def\convex{\Category{Convex}}

\def\sets{\Category{Set}}
\def\cats{\Category{Cat}}
\def\gradedalgebra{\Category{GradedAlgebra}}

\def\Cubical{\underline{\Category{\simp}}}
\def\Cube{\simp}
\def\simp{\Box}
\def\Cubic{C}
\def\cubictxt{cubic set}
\def\cubicstxt{cubic sets}

\def\Cubicaltxt{Cubical}
\def\Cubetxt{Cube}
\def\cubicaltxt{cubical}

\def\sitetxt{monoidal site}

\def\cubedef{0 \!\leq\! t_{1}, \cdots, t_{n} \!\leq\! 1}
\makeatletter
\def\proofname{{\it Proof}\,:\,\ }
\def\proofolname{{\it Outline of the proof}\,:\,\ }
\def\pproofname{{\it Proof of }}
\def\pproofolname{{\it Outline of the proof of }}
\def\aproofname{{\it. }}
\def\aproofolname{{\it. }}
\def\@proof[#1]{\pproofname{\it #1}\aproofname}
\def\@proofol[#1]{\pproofolname{\it #1}\aproofolname}
\renewenvironment{proof}{\par\vskip-1.0ex\noindent\@ifnextchar[\@proof{\proofname}}{{\unskip\nobreak\hfill{ \it\,\qedsymbol}}\par\vskip5pt}

\ifx\undefined\bysame\newcommand{\bysame}{\leavevmode\hbox to3em{\hrulefill}\,}\fi
\makeatother
\numberwithin{equation}{section}
\numberwithin{thm}{section}
\def\Dom{\operatorname{Dom}}

\def\Ker{\operatorname{Ker}}
\def\Im{\operatorname{Im}}

\def\id{\operatorname{id}}
\def\incl{\mbox{in}}
\def\proj{\mbox{pr}}
\def\smallproj{\mbox{\scriptsize pr}}
\def\midvert{\,\mathstrut\vrule\,}
\def\cardinal{\mathbb{N}_{0}}

\def\integral{\mathbb{Z}}

\def\real{\mathbb{R}}
\def\qedsymbol{\vbox{\hrule\hbox{\vrule height1.2ex\hskip1.2ex\vrule}\hrule}\hskip0.25ex} 
\def\der#1by#2{\frac{\operatorname{\mathit{d}}\hspace{-0.1mm}#1}{\operatorname{\mathit{d}}\hspace{-0.1mm}#2}}
\def\pder#1by#2{\frac{\operatorname{\partial}\hspace{-0.1mm}#1}{\operatorname{\partial}\hspace{-0.1mm}#2}}
\def\npder#1by#2times#3{\frac{\operatorname{{\partial\,}^{#3}}\hspace{-0.2mm}#1}{\operatorname{\partial}\hspace{-0.1mm}{#2}^{#3}}}
\def\diff#1{\ifx#1(\operatorname{\mathit{d}}\hspace{0.25mm}#1\else\operatorname{\mathit{d}}\hspace{-0.5mm}#1\fi}
\def\pdiff#1{\ifx#1(\operatorname{\partial}\hspace{0.25mm}#1\else\operatorname{\partial}\hspace{-0.25mm}#1\fi}
\title{Mayer-Vietoris sequence \\ for \\ differentiable/diffeological spaces}
\author[Iwase]{Norio IWASE}
\email[Iwase]{iwase@math.kyushu-u.ac.jp}
\address[Iwase]{Faculty of Mathematics,
 Kyushu University,
 Motooka 744,
 Fukuoka 819-0395, Japan}
\author[Izumida]{Nobuyuki IZUMIDA}
\email[Izumida]{isla.de.salsa@gmail.com}
\address[Izumida]{Puropera Corporation, Tomigaya 1-34-6, Shibuya, Tokyo 151-0063, Japan}
\date{\today}
\keywords{differentiable, diffeology, partition of unity, differential form, de Rham theory, singular cohomology.}
\subjclass[2010]{Primary 58A40, Secondary 58A03, 58A10, 58A12, 55N10}
\begin{document}
\baselineskip18pt
%
%
\begin{abstract}
The idea of a space with smooth structure is a generalization of an idea of a manifold.
K. T. Chen introduced such a space as a differentiable space in his study of a loop space to employ the idea of iterated path integrals \cite{Chen:73,Chen:75,Chen:77,Chen:86}.
Following the pattern established by Chen, J. M. Souriau \cite{Souriau:80} introduced his version of a space with smooth structure, which is called a diffeological space. 
These notions are strong enough to include all the topological spaces.
However, if one tries to show de Rham theorem, he must encounter a difficulty to obtain a partition of unity and thus the Mayer-Vietoris exact sequence in general.
In this paper, we introduce a new version of differential forms to obtain a partition of unity, the Mayer-Vietoris exact sequence and a version of de Rham theorem in general.
In addition, if we restrict ourselves to consider only CW complexes, we obtain de Rham theorem for a genuine de Rham complex, and hence the genuine de Rham cohomology coincides with the ordinary cohomology for a CW complex.
\end{abstract}
%
%
\allowdisplaybreaks
\maketitle
%
%

In this paper, we deal with both differentiable and diffeological spaces.
A differentiable space is introduced by K. T. Chen \cite{Chen:86} and a diffeological space is introduced by J. M. Souriau \cite{Souriau:80}.
Both of them are developed with an idea of a \textit{plot} -- a map from a \textit{domain}.

Let  $n \geqq 0$.
A non-void open set in $\real^{n}$ is called an \textit{open $n$-domain} or simply an \textit{open domain} and a convex set with non-void interior in $\real^{n}$ is called a \textit{convex $n$-domain} or simply a \textit{convex domain}.
We reserve the word `smooth' for `differentiable infinitely many times' in the ordinary sense.
More precisely, a map from an open or convex domain $A$ to an euclidean space is smooth on $A$, if it is smooth on $\Int{A}$ in the ordinary sense and all derivatives extend continuously and uniquely to $A$ (see A. Kriegl and P. W. Michor \cite{KM:14}).

\section{Differentiable/diffeological spaces}

Let us recall a concrete site given by Chen \cite{Chen:86} (see J. C. Baes and A. E. Hoffnung \cite{BH:11}).

\begin{defn}\label{defn:convex}
Let $\convex$ be the category of convex domains and smooth maps between them.
Then $\convex$ is a concrete site with Chen's coverage: a covering family on a convex domain is an open covering by convex domains. 
\end{defn}

On the other hand, a concrete site given by Souriau \cite{Souriau:80} is as follows.

\begin{defn}\label{defn:domain}
Let $\open$ be the category of open domains and smooth maps between them.
Then $\open$ is a concrete site with the usual coverage: a covering family on an open domain is an open covering by open domains.
\end{defn}

Let $\sets$ be the category of sets. 
A differentiable or diffeological space is as follows.

\begin{defn}[Differentiable space]\label{defn:chen-structure}
A differentiable space is a pair $(X,\mathcal{C}^{X})$ of a set $X$ and a contravariant functor $\mathcal{C}^{X} : \convex \to \sets$ such that 
\begin{enumerate}
\vspace{.5ex}
\item[(C0)]
For any $A \in \Object{\convex}$, $\mathcal{C}^{X}(A) \subset \Morphism{\sets}(A,X)$.
\vspace{.5ex}
\item[(C1)]
For any $x \in X$ and any $A \in \Object{\convex}$, $\mathcal{C}^{X}(A) \ni c_{x}$ the constant map.
\vspace{.5ex}
\item[(C2)]
Let $A \in \Object{\convex}$ with an open covering $A = \underset{\alpha\in\Lambda}{\cup} {\Int_{A}{B_{\alpha}}}, B_{\alpha} \in \Object{\convex}$.
If $P \in \Morphism{\sets}(A,X)$ satisfies that $P\vert_{B_{\alpha}} \in \mathcal{C}^{X}(B_{\alpha}))$ for all $\alpha\in\Lambda$, then $P \in \mathcal{C}^{X}(A)$.
\vspace{.5ex}
\item[(C3)]
For any $A, B \in \Object{\convex}$ and any $f \in \Morphism{\convex}(B,A)$, $\mathcal{C}^{X}(f) = f^{*} : \mathcal{C}^{X}(A) \to \mathcal{C}^{X}(B)$ is given by $f^{*}(P) = P{\circ}f \in \mathcal{C}^{X}(A)$ for any $P \in \mathcal{C}^{X}(A)$.
\end{enumerate}
\end{defn}
\begin{defn}[Diffeological space]\label{defn:souriau-structure}
A diffeological space is a pair $(X,\mathcal{D}^{X})$ of a set $X$ and a contravariant functor $\mathcal{D}^{X} : \open\to \sets$ such that 
\begin{enumerate}
\vspace{.5ex}
\item[(D0)]\label{defn:diffeology'-0}
For any $U \in \Object{\open}$, $\mathcal{D}^{X}(U) \subset \Map(U,X)$.
\vspace{.5ex}
\item[(D1)]\label{defn:diffeology'-1}
For any $x \in X$ and any $U \in \Object{\open}$, $\mathcal{D}^{X}(U) \ni c_{x}$ the constant map.
\vspace{.5ex}
\item[(D2)]\label{defn:diffeology-2}
Let $U \in \Object{\open}$ with an open covering $U = \underset{\alpha\in\Lambda}\cup {V_{\alpha}}, \ V_{\alpha} \in \Object{\open}$.
If $P \in \Morphism{\sets}(U,X)$ satisfies that $P|_{V_{\alpha}} \in \mathcal{D}^{X}(V_{\alpha})$ for all $\alpha \in \Lambda$, then $P \in \mathcal{D}^{X}(U)$.
\vspace{.5ex}
\item[(D3)]\label{defn:diffeology-3}
For any $U, V \in \Object{\open}$ and any $f \in \Morphism{\open}(V,U)$, $\mathcal{D}^{X}(f) = f^{*} : \mathcal{D}^{X}(V) \to \mathcal{D}^{X}(U)$ is given by $f^{*}(P)=P{\circ}f \in \mathcal{D}^{X}(V)$ for any $P \in \mathcal{D}^{X}(U)$.
\end{enumerate}
\end{defn}

From now on, $\mathcal{E}^{X} : \domain \to \sets$ stands for either $\mathcal{C}^{X} : \convex \to \sets$ or $\mathcal{D}^{X} : \open \to \sets$ to discuss about a differentiable space and a diffeological space simultaneously. 
\begin{defn}
A subset $O \subset X$ is open if, for any $P \in \mathcal{E}^{X}$\, ($\mathcal{E}=\mathcal{C}$ or $\mathcal{D}$), $P^{-1}(O)$ is open in $\Dom{P}$.
When any compact subset of $X$ is closed, we say $X$ is `weakly-separated'.
\end{defn}

\begin{defn}\label{defn:smooth-structure-1}
Let $(X,\mathcal{E}^{X})$ and $(Y,\mathcal{E}^{Y})$ be differentiable/diffeological spaces, $\mathcal{E}=\mathcal{C}$ or $\mathcal{D}$.
A map $f : X \to Y$ is differentiable, if there exists a natural transformation of contravariant functors $\mathcal{E}^{f} : \mathcal{E}^{X} \to \mathcal{E}^{Y}$ such that $\mathcal{E}^{f}(P)=f{\circ}P$.
The set of differentiable maps between $X$ and $Y$ is denoted by $C_{\mathcal{E}}^{\infty}(X,Y)$ or simply by $C^{\infty}(X,Y)$.
If further, $f$ is invertible with a differentiable inverse map, $f$ is said to be a diffeomorphism.
\end{defn}
Let us summarize the minimum notions from \cite{Chen:73,Chen:75,Chen:77,Chen:86,Souriau:80,BH:11,Wu:12,Iglesias:13,Haraguchi:14,Izumida:14} to build up de Rham theory in the category of differentiable or diffeological spaces as follows.

\begin{defn}[External algebra]
Let $T^{*}_{n}=\Hom(\real^{n},\real)=\overset{n}{\underset{i=1}{\oplus}}\real\diff{x_{i}}$, where $\{\diff{x_{i}}\}_{1 \leqq i \leqq n}$ is the dual basis to the standard basis $\{e_{i}\}_{1 \leqq i \leqq n}$ of $\real^{n}$.
We denote by $\wedge^{*}(T^{*}_{n})$ the exterior (graded) algebra on $\{\diff{x_{i}}\}$, where each $\diff{x_{i}}$ is of dimension $1$.
In particular, we have $\wedge^{0}(T^{*}_{n})\cong \wedge^{*}(T^{*}_{0})\cong\real$, $\wedge^{p}(T^{*}_{n})=0$ if $p\!<\!0$ and $\wedge^{p}(T^{*}_{n})\cong\wedge^{n-p}(T^{*}_{n})$ for any $p \in \integral$.
\end{defn}

The external algebra fits in with our categorical context as the following form.
\begin{defn}
\label{defn:souriau-structure-2}
A contravariant functor $\wedge^{p} : \domain \to \sets$ is given as follows:
\begin{enumerate}
\item
$\wedge^{p}(A)=\Morphism{\domain}(A,\wedge^{p}(T^{*}_{n}))$, for any convex $n$-domain $A$, 
\item
For a smooth map $f : B \to A$ in $\domain$, $\wedge^{p}(f)\!=\!f^{*} : \wedge^{p}(A) \to \wedge^{p}(B)$ is defined, for any $\omega=\underset{i_{1}{<}{\cdots}{<}i_{p}}{\textstyle\sum}a_{i_{1},{\cdots}\,,i_{p}}(\mathbold{x})$ $\diff{x_{i_{1}}}\wedge\cdots\wedge\diff{x_{i_{p}}} \in \wedge^{p}(A)$, as 
\par\vskip.5ex\noindent\hfil$
\begin{array}{l}
f^{*}(\omega) = \displaystyle\underset{j_{1}<{\cdots}<j_{p}}{\textstyle\sum}b_{j_{1},{\cdots},j_{p}}(\mathbold{y}){\cdot}\diff{y_{j_{1}}}\wedge\cdots\wedge\diff{y_{j_{p}}}, \ \mathbold{y} \in V,
\\[1.5ex] \,
b_{j_{1},{\cdots},j_{p}}(\mathbold{y}) =\displaystyle\!\!\! \underset{i_{1}<{\cdots}<i_{p}}{\textstyle\sum}a_{i_{1},{\cdots},i_{p}}(f(\mathbold{y})){\cdot}\pder{(x_{i_{1}},{\cdots},x_{i_{p}})}by{(y_{j_{1}},{\cdots},y_{j_{p}})},
\end{array}$\vspace{-.5ex}
\end{enumerate}
where $\displaystyle\pder{(x_{i_{1}},{\cdots},x_{i_{p}})}by{(y_{j_{1}},{\cdots},y_{j_{p}})}$ denotes the Jacobian determinant.
\end{defn}
\begin{defn}\label{defn:souriau-structure-3}
A natural transformation $\diff : \wedge^{p} \to \wedge^{p+1}$ is given as follows: for any domain $A$, $\diff{} : \wedge^{p}(A) \to \wedge^{p+1}(A)$ is defined, for any $\eta=a(\mathbold{x})\diff{x_{i_{1}}}{\wedge}\cdots{\wedge}\diff{x_{i_{p}}} \in \wedge^{p}(A)$, as 
\par\vskip.5ex\noindent\hfil$\displaystyle
\diff{\eta} = \underset{i}{\textstyle\sum}\,\pder{a_{i_{1},\cdots,i_{p}}}by{x_{i}}(\mathbold{x})\diff{x_{i}}{\wedge}\diff{x_{i_{1}}}{\wedge}\cdots{\wedge}\diff{x_{i_{p}}}.
$\hfil\par\vskip.5ex\noindent
Then the naturality is obtained using a strait-forward computation.
\end{defn}
A differential form is given in this context as follows.
\begin{defn}\label{defn:souriau-structure-4}
Let $(X,\mathcal{E}^{X})$ be a differentiable or diffeological space, $\mathcal{E}=\mathcal{C}$ or $\mathcal{D}$.
\begin{description}
\vskip.5ex
\item[(general)]
A differential $p$-form on 
$X$ is a natural transformation $\omega : \mathcal{E}^{X}$ $\to$ $\wedge^{p}$ given by $\{\omega_{A} : \mathcal{E}^{X}(A) \!\to\!\wedge^{p}(A)\,;\,A \!\in\! \Object{\domain}\}$ of contravariant functors $\mathcal{E}^{X},\,\wedge^{p} : \domain \to \sets$, in other words, $\omega$ satisfies $f^{*}(\omega_{B}(P))=f^{*}{\circ}\omega_{B}(P)=\omega_{A}{\circ}f^{*}(P)=\omega_{A}(P{\circ}f)$ for any map $f : A \to B$ in $\domain$ and a plot $P \in \mathcal{E}^{X}(B)$.
The set of differential $p$-forms on $X$ is denoted by $\Omega_{\mathcal{E}}^{p}(X)$ or simply by $\Omega^{p}(X)$.
We also denote $\Omega_{\mathcal{E}}^{\ast}(X) = \underset{p}{\oplus}\, \Omega_{\mathcal{E}}^{p}(X)$ or by $\Omega^{\ast}(X) = \underset{p}{\oplus}\, \Omega^{p}(X)$.
\item[(with compact support)]
A differential $p$-form with compact support on $X$ is a natural transformation $\omega = : \mathcal{E}^{X} \to \wedge^{p}(-)$ with a compact subset $K_{\omega} \subset X$ such that, for any $A \!\in\! \Object{\domain}$ and $P \in \mathcal{E}^{X}$, we have $\Supp{\omega_{A}(P)} \subset P^{-1}(K_{\omega})$.
The set of differential $p$-forms with compact support on $X$ is denoted by $\Omega_{\mathcal{E}_{c}}^{p}(X)$ or simply by $\Omega_{c}^{p}(X)$.
We also denote $\Omega_{\mathcal{E}_{c}}^{\ast}(X) = \underset{p}{\oplus}\, \Omega_{\mathcal{E}_{c}}^{p}(X)$ or $\Omega_{c}^{\ast}(X) = \underset{p}{\oplus}\, \Omega_{c}^{p}(X)$.
\end{description}
\end{defn}
\begin{expl}
We have $\Omega^{\ast}(\{\ast\})\cong\real$ and $\Omega_{c}^{\ast}(\{\ast\})\cong\real$.
\end{expl}
\begin{defn}[External derivative]\label{defn:souriau-structure-5}
The external derivative of a differential $p$-form $\omega$ on a differentiable/diffeological space $X$ is a differential $p{+}1$-form $\diff\omega$ given by $(\diff\omega)_{A}=\diff{\circ}\omega_{A}$ for any $A \in \Object{\domain}$. 
If, further we assume $\omega \in \Omega_{c}^{p}(X)$, we clearly have $\diff\omega \in \Omega_{c}^{p+1}(X)$.
Thus the external derivative induces endomorphisms of $\Omega^{\ast}(X)$ and $\Omega_{c}^{\ast}(X)$.
\end{defn}

The category of differentiable or diffeological spaces and differentiable maps is denoted by $\differentiable$ or $\diffeological$, respectvely.
By \cite{Souriau:80}, \cite{Chen:86} and \cite{BH:11}, we know $\differentiable$ and $\diffeological$ are cartesian closed, complete and cocomplete.

\begin{defn}\label{defn:chen-induced-map}
Let $f : (X,\mathcal{E}^{X}) \to (Y,\mathcal{E}^{Y})$ be a differentiable map, $\mathcal{E}=\mathcal{C}$ or $\mathcal{D}$.
\begin{enumerate}
\item
We obtain a homomorphism $f^{\sharp} : \Omega^{p}(Y) \to \Omega^{p}(X)$:
let $\omega \in \Omega^{p}(Y)$. Then 
$(f^{\sharp}\omega)_{A}(P)=\omega_{A}(f{\circ}P) \  \text{for any \ $P \in \mathcal{E}^{X}(A)$ and $A \in \Object{\domain}$.}
$\vspace{.5ex}
\item
If a differentiable map $f$ is proper, then we have $f^{\sharp}(\Omega^{p}_{c}(Y)) \subset \Omega^{p}_{c}(X)$ by taking $K_{f^{\sharp}\omega}=f^{-1}(K_{\omega})$ \ for any $\omega \in \Omega^{p}_{c}(Y)$.
\end{enumerate}
\end{defn}

\begin{defn}\label{defn:chen-induced-map-1}
For an inclusion $j : U \hookrightarrow X$ of an open set $U$ into a weakly-separated differentiable/diffeological space $X$, a homomorphism $j_{\sharp} : \Omega^{p}_{c}(U) \to \Omega^{p}_{c}(X)$ is defined as follows: for any $\omega \in \Omega^{p}_{c}(U)$, $j_{\sharp}\omega \in \Omega^{p}_{c}(X)$ is given, for $n$-domain $B$ and $Q \in \mathcal{E}^{X}(B)$, by
\par\vskip1.5ex\noindent
\hfil$
\left\{\!\!\begin{array}{l}
(j_{\sharp}\omega)_{B}(Q)\vert_{A} = \omega_{A}(Q\vert_{A}), \quad \text{if \ $A$ is an open $n$-domain in $Q^{-1}(U)$,}
\\[1ex]
(j_{\sharp}\omega)_{B}(Q)\vert_{A} = 0,
\quad \text{if \ $A$ is an open $n$-domain in $B \!\smallsetminus\! Q^{-1}(K_{\omega})$} 
\end{array}\right.
$\hfil
\par\vskip1.5ex\noindent
with $K_{j_{\sharp}\omega} = K_{\omega} \subset U \subset X$. Here, $\{Q^{-1}(U),B \!\smallsetminus\! Q^{-1}(K_{\omega})\}$ is an open covering of $B$.
\end{defn}
\begin{rem}
In Definition \ref{defn:chen-induced-map-1}, the map $j_{\sharp}$ induced from an inclusion $j : U \hookrightarrow X$ satisfies that $(j_{\sharp}\omega)_{B}(j{\circ}Q) = \omega_{B}(Q)$ for any $B \in \Object{\domain}$ and $Q \in \mathcal{E}^{U}(B)$.
\end{rem}

\begin{prop}
There is an isomorphism $\Phi : \Omega^{0}(X) \cong C^{\infty}(X,\real)$ such that $\Phi(\omega){\circ}f = \Phi(f^{\sharp}(\omega))$ for any $\omega \in \Omega^{0}(X)$ and $f \in C^{\infty}(Y,X)$.
\end{prop}
\begin{proof}
Firstly, we define a homomorphism $\Phi : \Omega^{0}(X) \to \Morphism{\sets}(X,\real)$ by $\Phi(\omega)(x)=\omega_{\{\ast\}}(c_{x})(\ast)$ $\in$ $\real$ for any $\omega \in \Omega^{0}(X)$ and $x \in X$.
By definition, $\Phi$ clearly is a homomorphism.

Secondly, we show $\Im{\Phi} \subset C^{\infty}(X,\real)$.
For any $n$-domain $A$ and $P \in \mathcal{E}^{X}(A)$, we have $\omega_{A}(P) : A \to \wedge^{0}(T^{*}_{n}) = \real$.
Hence for any $\mathbold{x} \in A$, we have $P{\circ}c_{\mathbold{x}}=c_{x} \in \mathcal{E}^{X}(\{\ast\})$ where $x=P(\mathbold{x}) \in X$, and hence we have $\omega_{A}(P)(\mathbold{x})=\omega_{A}(P){\circ}c_{\mathbold{x}}(\ast)=\omega_{\{\ast\}}(P{\circ}c_{\mathbold{x}})(\ast)=\omega_{\{\ast\}}(c_{x})(\ast)=\Phi(\omega)(x)=\Phi(\omega){\circ}P(\mathbold{x})$, $\mathbold{x} \in A$.
Thus we have $\omega_{A}(P)=\Phi(\omega){\circ}P$ for any $A \in \Object{\domain}$ and $P \in \mathcal{E}^{X}(A)$, and hence $\Phi(\omega) : X \to \real$ is a differentiable map.
Moreover, for any differentiable map $f : Y \to X$, we have $\Phi(f^{\sharp}\omega)(x) = (f^{\sharp}\omega)_{\{\ast\}}(c_{x})(\ast) = \omega_{\{\ast\}}(f{\circ}c_{x})(\ast) = \omega_{\{\ast\}}(c_{f(x)})(\ast) = \Phi(\omega){\circ}f(x)$, and hence we obtain $\Phi(f^{\sharp}\omega) = \Phi(\omega){\circ}f$.

Thirdly, by the formula $\omega_{A}(P)=\Phi(\omega){\circ}P$ for any $A \in \Object{\domain}$ and $P \in \mathcal{E}^{X}(A)$, $\omega$ is completely determined by $\Phi(\omega)$, and hence $\Phi$ is a monomorphism.

Finally, for any differentiable map $f : X \to \real$, we have a $0$-form $\omega$ by $\omega_{A}(P)=f{\circ}P$ for any $A \in \Object{\domain}$ and $P \in \mathcal{E}^{X}(A)$, which also implies $\Phi(\omega)=f$.
Thus $\Phi$ is an epimorphism, and it completes the proof of the proposition.
\end{proof}

\begin{defn}
Let $X=(X,\mathcal{E})$ be  a differentiable/diffeological space, $\mathcal{E}=\mathcal{C}$ or $\mathcal{D}$.
\begin{description}
\item[de Rham cohomology]
$H_{\mathcal{E}}^{p}(X) = \displaystyle\frac{Z_{\mathcal{E}}^{p}(X)}{B_{\mathcal{E}}^{p}(X)}$,
\par\vskip.5ex\noindent
where $Z_{\mathcal{E}}^{p}(X) = \Ker{d}\cap \Omega_{\mathcal{E}}^{p}(X)$
and $B_{\mathcal{E}}^{p}(X) = \diff (\Omega_{\mathcal{E}}^{p}(X))$.
\vskip1ex
\item[de Rham cohomology with compact support] 
$H_{\mathcal{E}_{c}}^{p}(X) = \displaystyle\frac{Z_{\mathcal{E}_{c}}^{p}(X)}{B_{\mathcal{E}_{c}}^{p}(X)}$,
\par\vskip.5ex\noindent
where $Z_{\mathcal{E}_{c}}^{p}(X) = \Ker{d}\cap \Omega_{\mathcal{E}_{c}}^{p}(X)$
and $B_{\mathcal{E}_{c}}^{p}(X) = \diff (\Omega_{\mathcal{E}_{c}}^{p}(X))$.
\end{description}
From now on, we often abbreviate as $H^{p}(X) = H_{\mathcal{E}}^{p}(X)$, $H_{c}^{p}(X) = H_{\mathcal{E}_{c}}^{p}(X)$ and so on.
\end{defn}
\begin{rem}\label{rem:manifold-C}
We have $H_{\mathcal{E}_{c}}^{p}(M) \cong H_{dR}^{p}(M)$ and $H_{\mathcal{E}_{c}}^{p}(M) \cong H_{dR_{c}}^{p}(M)$ for a manifold $M$, where we denote by $H_{dR}^{p}(M)$ ($H_{dR_{c}}^{p}(M)$) the de Rham cohomology (with compact support).
\end{rem}

\begin{prop}
Let $(X,\mathcal{E}^{X})$ and $(Y,\mathcal{E}^{Y})$ be differentiable/diffeological spaces.
\begin{enumerate}
\item
For a differentiable map $f : X \to Y$, the homomorphism $f^{\sharp} : \Omega^{\ast}(Y) \to \Omega^{\ast}(X)$ induces a homomorphism $f^{*} : H^{\ast}(Y) \to H^{\ast}(X)$.
\vskip.5ex
\item
If a differentiable map $f : X \to Y$ is proper, then the homomorphism $f^{\sharp} : \Omega_{c}^{\ast}(Y) \to \Omega_{c}^{\ast}(X)$ induces a homomorphism $f^{*} : H^{\ast}_{c}(Y) \to H^{\ast}_{c}(X)$.
\end{enumerate}
\end{prop}

\begin{thm}
The de Rham cohomologies determines contravariant functors $H_{\mathcal{C}}^{\ast} : \differentiable \to \gradedalgebra$ and $H_{\mathcal{D}}^{\ast} : \diffeological \to \gradedalgebra$. 
\end{thm}

\begin{prop}
Let $(X,\mathcal{E}^{X})$ be a weakly-separated differentiable/diffeological space and $U$ an open set in $X$.
Then the homomorphism $j_{\sharp} : \Omega_{c}^{\ast}(U) \to \Omega_{c}^{\ast}(X)$ induced from the canonical inclusion $j : U \hookrightarrow X$ induces a homomorphism $j_{*} : H_{c}^{*}(U) \to H_{c}^{*}(X)$.
\end{prop}

\begin{thm}[\cite{Chen:86}, \cite{Souriau:80}]
If two differentiable maps $f_{0}, \ f_{1} : X \to Y$ between differentiable/diffeological spaces are homotopic in $C_{\mathcal{E}}^{\infty}(X,Y)$, $\mathcal{E}=\mathcal{C}$ or $\mathcal{D}$, i.e., there is a differentiable map $f : I \to C_{\mathcal{E}}^{\infty}(X,Y)$ such that $f(t)=f_{t}$, $t = 0,1$, then we obtain
\par\vskip.5ex\noindent\hfil
$f^{*}_{0} = f^{*}_{1} : H_{\mathcal{E}}^{\ast}(Y) \to H_{\mathcal{E}}^{\ast}(X)$.
\par\vskip.5ex
\end{thm}

\begin{thm}
By definition, we clearly have $H^{\ast}_{\mathcal{E}}(\underset{\alpha}{\coprod}X_{\alpha}) = \underset{\alpha}{\prod}H^{\ast}_{\mathcal{E}}(X_{\alpha})$, $\mathcal{E}=\mathcal{C}$ or $\mathcal{D}$.
\end{thm}

\begin{expl}\label{expl:one-point-set-C}
For a differentiable/diffeological space $(\{\ast\},\mathcal{E}^{*})$ with $\mathcal{E}^{*}(A) = \{c_{\ast}\}$ for any $A \in \Object{\domain}$, we have $H^{0}(X) = \Omega^{0}(X)=\real$ and $H^{p}(X) = \Omega^{p}(X)=0$ if $p\not=0$.
\end{expl}

\section{Mayer-Vietoris sequence for differentiable spaces}

\begin{defn}[partition of unity]
Let $(X,\mathcal{E}^{X})$ be a differentiable/diffeological space and \,$\mathcal{U}$ an open covering of $X$.
A set of \,$0$-forms $\mathbold{\rho}=\{\rho^{U}\,;\,U \!\in\! \mathcal{U}\}$ is called a partition of unity belonging to  \,$\mathcal{U}$, if, for any $A \in \Object{\domain}$ and $P \in \mathcal{E}^{X}(A)$, $\Supp{\rho^{U}_{A}(P)} \subset P^{-1}(U)$ and $\underset{U \in \mathcal{U}}{\sum}\, \rho^{U}_{A}(\mathbold{x}) = 1$, $\mathbold{x} \in A$.
If further there is a family $\{G_{U}\,;\,U \!\in\! \mathcal{U}\}$ of closed sets in $X$ such that, $\Supp{\rho^{U}_{A}(P)} \subset P^{-1}(G_{U})$ for any $A$ and $P$ above, then we say that $\mathbold{\rho}$ is `normal'.
\end{defn}

The above definition of a partition of unity using the notion of \,$0$-form first appeared in Izumida \cite{Izumida:14} which was essentially the same as the one in Haraguchi \cite{Haraguchi:14} using the notion of a differentiable function, since a differential $0$-form is a differentiable function, if we adopt the usual definition of \,$0$-form.
We introduce a special kind of open coverings as follows.

\begin{defn}[Nice covering]\label{defn:nice-open-covering}
Let $X$ be a differentiable space.
An open covering \,$\mathcal{U}$ of $X$ is \textit{nice}, if there is a partition of unity $\{\rho^{U}_{A} : A \to I=[0,1]\,;\,U \!\in\! \mathcal{U}\}$ belonging to \,$\mathcal{U}$, i.e., $\{\rho^{U}\}$ are differential $0$-forms with $\Supp{\rho_{A}^{U}(P)} = \Clo{(\rho_{A}^{U}(P)^{-1}(I \!\smallsetminus\! \{0\}))} \subset P^{-1}(U)$, $U \!\in\! \,\mathcal{U}$ satisfying $\underset{U \in \mathcal{U}}{\sum}\, \rho_{A}^{U}(P)(\mathbold{x}) = 1$ for any $\mathbold{x} \!\in\! A$, where $\rho_{A}^{U}(P)(\mathbold{x})\not=0$ for finitely many $U$.
\end{defn}

\begin{thm}[see \cite{Haraguchi:14} or \cite{Izumida:14}]\label{thm:nice-open-covering-1}
Let \,$\mathcal{U}=\{U_{1},U_{2}\}$ be a nice open covering of a differentiable/diffeological space $(X,\mathcal{E}^{X})$ with a partition of unity $\{\rho^{(1)},\rho^{(2)}\}$ belonging to \,$\mathcal{U}$.
Then $i_{t} : U_{1} \cap U_{2} \hookrightarrow U_{t}$ and $j_{t} : U_{t} \hookrightarrow X$, $t=1, 2$, induce homomorphisms $\psi^{\natural} : \Omega^{p}(X) \to \Omega^{p}(U_{1}) \oplus \Omega^{p}(U_{2})$ and $\phi^{\natural} : \Omega^{p}(U_{1}) \oplus \Omega^{p}(U_{2}) \to \Omega^{p}(U_{1} \cap U_{2})$ by $\psi^{\natural}(\omega)$ $=$ $i_{1}^{\sharp}\omega{\oplus}i_{2}^{\sharp}\omega$ and $\phi^{\natural}(\eta_{1}{\oplus}\eta_{2})$ $=$ $j_{1}^{\sharp}\eta_{1}\!-\!j_{2}^{\sharp}\eta_{2}$, and the following sequence is exact.
\par\vskip1.5ex\noindent\hfil
$\begin{array}{l}
H^{0}(X) \rightarrow \cdots \rightarrow H^{p}(X) \xrightarrow{\psi^{*}} H^{p}(U_{1}){\oplus}H^{p}(U_{2}) \xrightarrow{\phi^{*}} H^{p}(U_{1} \cap U_{2})
\\[1ex]\qquad\qquad\qquad
\rightarrow H^{p+1}(X) \xrightarrow{\psi^{*}} H^{p+1}(U_{1}){\oplus}H^{p+1}(U_{2}) \xrightarrow{\phi^{*}} H^{p+1}(U_{1} \cap U_{2}) \rightarrow \cdots,
\end{array}$\hfil
\par\vskip1.5ex\noindent
where $\psi^{*}$ and $\phi^{*}$ are induced from $\psi^{\natural}$ and $\phi^{\natural}$.
\end{thm}
\begin{proof}
Let $U_{0}=U_{1} \cap U_{2}$.
We show that the following sequence is short exact.
\par\vskip1ex\noindent
\hfil$
0 \longrightarrow \Omega^{p}(X) 
\overset{\psi^{\natural}}\longrightarrow \Omega^{p}(U_{1}) \oplus \Omega^{p}(U_{2}) 
\overset{\phi^{\natural}}\longrightarrow \Omega^{p}(U_{0}) \longrightarrow 0.
$\hfil
\par\vskip1ex\noindent
\begin{description}
\item[(exactness at $\Omega^{p}(X)$)]
Assume $\psi^{\natural}(\omega)=0$, and so $j_{t}^{\sharp}\omega=0$ for $t\!=\!1, 2$.
For any $A \in \Object{\domain}$ and $P \in \mathcal{E}^{X}(A)$, we define $P_{t} : P^{-1}(U_{t}) \to U_{t}$, $t\!=\!1, 2$ by $P_{t}(\mathbold{x}) = P(\mathbold{x})$ for any $\mathbold{x} \in P^{-1}(U_{t})$, so that $P\vert_{P^{-1}(U_{t})}=j_{t}{\circ}P_{t}$ for $t\!=\!1, 2$.
Then, for any $\mathbold{x} \in A$, there is an open subset $A_{\mathbold{x}} \in \Object{\domain}$ of $A$ such that $\mathbold{x} \in A_{\mathbold{x}} \subset P^{-1}(U_{t})$ for $t\!=\!1$ or $2$.
In each case, we have $\omega_{A}(P) \vert_{A_{\mathbold{x}}}= \omega_{A_{\mathbold{x}}}(P\vert_{A_{\mathbold{x}}}) = \omega_{A_{\mathbold{x}}}(P\vert_{P^{-1}(U_{t})}\vert_{A_{\mathbold{x}}}) 
= \omega_{A_{\mathbold{x}}}(j_{t}{\circ}P_{t}\vert_{A_{\mathbold{x}}}) = (j_{t}^{\sharp}\omega)_{A_{\mathbold{x}}}(P_{t}\vert_{A_{\mathbold{x}}}) = 0$, and hence $\omega_{A}(P)\vert_{A_{\mathbold{x}}}=0$ for any $\mathbold{x} \in A$. Thus $\omega_{A}(P)=0$ for any $A$ and $P$, which implies that $\omega=0$.
Thus $\psi^{\natural}$ is monic.
\vspace{1ex}
\item[(exactness at $\Omega^{p}(U_{1}) \oplus \Omega^{p}(U_{2})$)]
Assume $\phi^{\natural}(\eta^{(1)}{\oplus}\eta^{(2)})=0$, and so $i_{1}^{\sharp}\eta^{(1)}=i_{2}^{\sharp}\eta^{(2)}$.
Then we construct $\omega \in \Omega^{p}(X)$ as follows.
For any $A \in \Object{\domain}$ and $P \in \mathcal{E}^{X}(A)$, $\{P^{-1}(U_{t})\,;\,t\!=\!1,2\}$ is an open covering of $A$, and for $t\!=\!0, 1, 2$ we obtain $P_{t} : P^{-1}(U_{t}) \to U_{t}$ given by $P_{t}(\mathbold{x}) = P(\mathbold{x})$ for any $\mathbold{x} \!\in\! P^{-1}(U_{t})$, so that $P_{t}\vert_{P^{-1}(U_{0})}=i_{t}{\circ}P_{0}$ for $t\!=\!1, 2$.
For any $\mathbold{x} \in A$, there is an open subset $A_{\mathbold{x}} \in \Object{\domain}$ of $A$ such that $\mathbold{x} \in A_{\mathbold{x}} \subset P^{-1}(U_{t})$ for $t\!=\!1$ or $2$.
Using it, we define $\omega_{A}(P)(\mathbold{x})$ $=$ $\eta^{(t)}_{A_{\mathbold{x}}}(P\vert_{A_{\mathbold{x}}})(\mathbold{x})$ for any $\mathbold{x} \in A$.
In case when $A_{\mathbold{x}} \!\subset\! A_{0} \!=\! A_{1}{\cap} A_{2}$, we have $\eta^{(1)}_{A_{\mathbold{x}}}(P_{1}\vert_{A_{\mathbold{x}}})$ 
$=$ $\eta^{(1)}_{A_{\mathbold{x}}}(i_{1}{\circ}P_{0}\vert_{A_{\mathbold{x}}})$ 
$=$ $(i_{1}^{\sharp}\eta^{(1)})_{A_{\mathbold{x}}}(P_{0}\vert_{A_{\mathbold{x}}})$ 
$=$ $(i_{2}^{\sharp}\eta^{(2)})_{A_{\mathbold{x}}}(P_{0}\vert_{A_{\mathbold{x}}})$ 
$=$ $\eta^{(2)}_{A_{\mathbold{x}}}(i_{2}{\circ}P_{0}\vert_{A_{\mathbold{x}}})$ 
$=$ $\eta^{(2)}_{A_{\mathbold{x}}}(P_{2}\vert_{A_{\mathbold{x}}})$, and hence $\eta^{(1)}_{A_{\mathbold{x}}}(P_{1}\vert_{A_{\mathbold{x}}})$ $=$ $\eta^{(2)}_{A_{\mathbold{x}}}(P_{2}\vert_{A_{\mathbold{x}}})$.
It implies that $\omega$ is well-defined and $\psi^{\natural}(\omega) = \eta^{(1)}{\oplus}\eta^{(2)}$.
\vspace{1ex}
The converse is clear and we obtain $\Ker{\phi^{\natural}}=\Im{\psi^{\natural}}$.
\item[(exactness at $\Omega^{p}(U_{0})$)]
Assume $\kappa \in \Omega^{p}(U_{0})$. 
Then we define $\kappa^{(t)} \in \Omega^{p}(U_{t})$, $t\!=\!1, 2$ defined as follows.
For any $A_{t} \in \Object{\domain}$ and a plot $P_{t} : A_{t} \to U_{t}$, we define $\kappa^{(t)}_{A_{t}}(P_{t})(\mathbold{x})$ by $(-1)^{t-1}\rho^{3-t}_{A_{t}}(P_{t})(\mathbold{x}){\cdot}\kappa_{A_{t}}(P_{t})(\mathbold{x})$ if $\mathbold{x} \in P_{t}^{-1}(U_{3-t})$ and by $0$ if $\mathbold{x} \not\in P_{t}^{-1}(\Supp{\rho_{A_{t}}^{3-t}(P_{t})})$. 
Then we see that $\kappa^{(t)}$ is well-defined differential $p$-form on $U_{t}$ and $i_{1}^{\sharp}\kappa^{(1)} \!-\! i_{2}^{\sharp}\kappa^{(2)} = \kappa$, and hence $\phi^{\natural}(\kappa^{(1)}{\oplus}\kappa^{(2)}) = \kappa$.
Thus $\phi^{\natural}$ is an epimorphism.
\end{description}
Since $\psi^{\natural}$ and $\phi^{\natural}$ are clearly cochain maps, we obtain the desired long exact sequence.
\end{proof}

Let us turn our attention to the differential forms with compact support.
\begin{thm}[see \cite{Haraguchi:14} or \cite{Izumida:14}]\label{thm:nice-open-covering-2}
Let $(X,\mathcal{E}^{X})$ be a {weakly-separated} differentiable/diffeological space and \,$\mathcal{U}=\{U_{1},U_{2}\}$ a nice open covering of $X$ with a normal partition of unity $\{\rho^{(1)},\rho^{(2)}\}$ belonging to \,$\mathcal{U}$.
Then $i_{t} : U_{1} \cap U_{2} \hookrightarrow U_{t}$ and $j_{t} : U_{t} \hookrightarrow X$, $t=1, 2$, induce homomorphisms $\phi_{\natural}$ $:$ $\Omega_{c}^{p}(U_{1} \cap U_{2}) \to \Omega_{c}^{p}(U_{1}) {\oplus} \Omega_{c}^{p}(U_{2})$ and $\psi_{\natural} : \Omega_{c}^{p}(U_{1}) {\oplus} \Omega_{c}^{p}(U_{2}) \to \Omega_{c}^{p}(X)$ by $\phi_{\natural}(\omega) = i_{1\sharp}\omega{\oplus}i_{2\sharp}\omega$ and $\psi_{\natural}(\eta_{1}{\oplus}\eta_{2}) = j_{1\sharp}\eta_{1}\!-\!j_{2\sharp}\eta_{2}$, and the following sequence is exact.
\par\vskip1ex\noindent\hfil
$\begin{array}{l}
H_{c}^{0}(U_{1} \cap U_{2}) \rightarrow \cdots \rightarrow H_{c}^{p}(U_{1} \cap U_{2}), 
\xrightarrow{\phi_{*}} H_{c}^{p}(U_{1}){\oplus}H_{c}^{p}(U_{2}) 
\xrightarrow{\psi_{*}} H_{c}^{p}(X)
\\[1ex]\qquad\qquad\qquad
\rightarrow H_{c}^{p+1}(U_{1} \cap U_{2}) 
\xrightarrow{\phi_{*}} H_{c}^{p+1}(U_{1}){\oplus}H_{c}^{p+1}(U_{2}) 
\xrightarrow{\psi_{*}} H_{c}^{p+1}(X) \rightarrow \cdots,
\end{array}$\hfil
\par\vskip1.5ex\noindent
where $\psi_{*}$ and $\phi_{*}$ are induced from $\psi_{\natural}$ and $\phi_{\natural}$.
\end{thm}
\begin{proof}
Let $U_{0}=U_{1} \cap U_{2}$.
We show that the following sequence is short exact.
\par\vskip.5ex\noindent
\hfil$
0 \longrightarrow \Omega_{c}^{p}(U_{0}) 
\overset{\phi_{\natural}}\longrightarrow \Omega_{c}^{p}(U_{1}) \oplus \Omega_{c}^{p}(U_{2}) 
\overset{\psi_{\natural}}\longrightarrow \Omega_{c}^{p}(X) \longrightarrow 0.
$\hfil
\par\vskip.5ex\noindent
\begin{description}
\item[(exactness at $\Omega_{c}^{p}(U_{0})$)]
Assume $\phi_{\natural}(\omega)=0$.
Then $i_{1\sharp}(\omega) = i_{2\sharp}(\omega) = 0$.
Since $i_{1\sharp}(\omega)$ is an extension of $\omega$, we obtain $\omega=0$.
Thus $\phi_{\natural}$ is a monomorphism.
\vspace{.5ex}
\item[(exactness at $\Omega_{c}^{p}(U_{1}) \oplus \Omega_{c}^{p}(U_{2})$)]
Assume $\psi_{\natural}(\eta^{(1)}{\oplus}\eta^{(2)})=0$.
By definition, we have $j_{1\sharp}(\eta^{(1)})=j_{2\sharp}(\eta^{(2)})$.
For any $A \in \Object{\domain}$ and $P \in \mathcal{E}^{X}(A)$, we have 
$j_{1\sharp}(\eta^{(1)})_{A}(P)$ $=$ $j_{2\sharp}(\eta^{(2)})_{A}(P)$.
So, for any $B \in \Object{\domain}$ and a plot $Q : B \to U_{0}$, $\eta^{(1)}_{B}(i_{1}{\circ}Q) = j_{1}^{\sharp}\eta^{(1)}_{B}(j_{1}{\circ}i_{1}{\circ}Q) = j_{2}^{\sharp}\eta^{(2)}_{B}(j_{2}{\circ}i_{2}{\circ}Q) = \eta^{(2)}_{B}(i_{2}{\circ}Q)$.
So we define $\eta^{(0)} \in \Omega^{p}(U_{0})$ by $\eta^{(0)}_{B}(Q) = \eta^{(1)}_{B}(i_{1}{\circ}Q)=\eta^{(2)}_{B}(i_{2}{\circ}Q)$.
On the other hand, $K_{j_{t\sharp}\eta^{(t)}} = K_{\eta^{(t)}}$ by definition, and hence we obtain $\Supp{\eta^{(0)}_{B}(Q)} = \Supp{\eta^{(1)}_{B}(i_{1}{\circ}Q)} = \Supp{\eta^{(2)}_{B}(i_{2}{\circ}Q)}$ $\subset$ $Q^{-1}(K_{\eta^{(1)}} \cap K_{\eta^{(2)}})$.
Then $\eta^{(0)} \in \Omega_{c}^{p}(U_{0})$ for $K_{\eta^{(0)}}=K_{\eta^{(1)}} \cap K_{\eta^{(2)}}$ is compact. 
\vspace{.5ex}
\item[(exactness at $\Omega_{c}^{p}(X)$)]
Assume $\kappa \in \Omega_{c}^{p}(X)$.
For any $A_{t} \in \Object{\domain}$ and a plot $P_{t} : A_{t} \to U_{t}$, we define $\kappa^{(t)}_{A_{t}}(P_{t})(\mathbold{x})$ by $(-1)^{t-1}\rho^{(t)}_{A_{t}}(P_{t})(\mathbold{x}){\cdot}\kappa_{A_{t}}(j_{t}{\circ}P_{t})(\mathbold{x})$ if $\mathbold{x} \in P_{t}^{-1}(U_{0})$ and by $0$ if $\mathbold{x} \not\in \Supp{\rho_{A_{t}}^{(t)}}(P_{t})$.
Then $\kappa^{(t)}$ is a well-defined differential $p$-form on $U_{t}$ with compact support $K_{\kappa^{(t)}} = K_{\kappa} \cap G_{U_{t}}$ in $U_{t}$ and $j_{1}^{\sharp}\kappa^{(1)} \!-\! j_{2}^{\sharp}\kappa^{(2)} = \kappa$, and hence we have $\psi_{\natural}(\kappa^{(1)}{\oplus}\kappa^{(2)}) = \kappa$.
Thus $\psi_{\natural}$ is an epimorphism.
\end{description}
Since $\phi_{\natural}$ and $\psi_{\natural}$ are clearly cochain maps, we obtain the desired long exact sequence.
\end{proof}

\section{{\Cubetxt} Category}

\begin{defn}\label{defn:cubic-category}
a concrete {\sitetxt} ${\Cubical}$ is defined as follows:
\begin{description}
\vspace{.5ex}
\item[Object]
$\Object{\Cubical} = \{\underline{0}, \underline{1}, \underline{2}, \cdots\} \approx \cardinal$, \quad 
$\underline{n}=\simp_{L}^{n} := \simp^{n} \cap L$, \\where\vspace{.5ex}
 $\simp^{n} \!=\! \{(t_{1},\ldots,t_{n}) \,;\,\cubedef\}$ and $L \!=\! \integral^{n} \!\subset\! \real^{n}$ is an integral lattice.
\item[Morphism]
$\Morphism{\Cubical}$ is generated by the following sets of morphisms.
\begin{description}
\vskip.5ex
\item[boundary]
$\partial^{\epsilon}_{i} : \underline{n} \to \underline{n{+}1}$, $\epsilon \!\in\! \dot{I}\!=\!\{0,1\}, 1 \!\leq\! i \!\leq\! n{+}1$, $n \!\geq\! 0$, given by 
\par\vskip.5ex$\textstyle%
\partial^{\epsilon}_{i}(\mathbold{t}) = (t_{1},\ldots,t_{i-1},\epsilon,t_{i+1},\ldots,t_{n})$ \ for \ $\mathbold{t} = (t_{1},{\cdots},t_{n}) \in \simp_{L}^{n}$,
\vskip1ex
\item[degeneracy]
$\varepsilon_{i} : \underline{n{+}1} \to \underline{n}$, \  
$1 \!\leq\! i \!\leq\! n{+}1, \ n \in \cardinal$ given by 
\par\vskip.5ex$\textstyle%
\varepsilon_{i}(\mathbold{t})=(t_{1},{\cdots},t_{i-1},t_{i+1},{\cdots},t_{n+1}), \ \ \mathbold{t}=(t_{1},{\cdots},t_{n+1}) \in \simp_{L}^{n},
$\hfil\par
\vskip.5ex
\end{description}
which satisfies the following relations.
\par\vskip1ex\noindent
\begin{inparaenum}
\item\label{prop:boundary-degeneracy-1}
$\displaystyle
\partial^{\epsilon'}_{j}{\circ}\partial^{\epsilon}_{i} = \begin{cases}
\partial^{\epsilon}_{i}{\circ}\partial^{\epsilon'}_{j-1} & \text{if} \ i \!<\! j 
\\[1ex]
\partial^{\epsilon}_{i+1}{\circ}\partial^{\epsilon'}_{j} & \text{if} \ i \!\geq\! j 
\end{cases}$\qquad
\item\label{prop:boundary-degeneracy-2}
$\displaystyle
\varepsilon_{j}{\circ}\varepsilon_{i} = \begin{cases}
\varepsilon_{i}{\circ}\varepsilon_{j+1} & \text{if} \ i \leq j \\[1ex]
\varepsilon_{i-1}{\circ}\varepsilon_{j} & \text{if} \ i > j
\end{cases}$
\par\vspace{1.5ex}\noindent
\item\label{prop:boundary-degeneracy-3}
$\displaystyle
\partial^{\epsilon'}_{j}{\circ}\varepsilon_{i} = \begin{cases}
\varepsilon_{i+1}{\circ}\partial^{\epsilon'}_{j} & \text{if} \ i \geq j \\[1ex]
\varepsilon_{i}{\circ}\partial^{\epsilon'}_{j+1} & \text{if} \ i < j
\end{cases}$\qquad
\item\label{prop:boundary-degeneracy-4}
$\displaystyle
\varepsilon_{j}{\circ}\partial^{\epsilon}_{i} = \begin{cases}
\partial^{\epsilon}_{i-1}{\circ}\varepsilon_{j} & \text{if} \ i > j \\[.5ex]
\partial^{\epsilon}_{i}{\circ}\varepsilon_{j-1} & \text{if} \ i < j \\[.5ex]
\id & \text{if} \ i = j
\end{cases}$
\end{inparaenum}
\end{description}
\end{defn}
Since $\simp^{n}_{L} = \simp^{n} \cap L \subset \real^{n}$, we can extend the boundaries and the degeneracies as smooth maps $\partial^{\epsilon}_{i} : \real^{n} \to \real^{n+1}$ and $\varepsilon_{i} : \real^{n+1} \to \real^{n}$.
Let ${\Cubical} : {\Cubical} \rightarrow \convex$ be the covariant functor defined by ${\Cubical}(\underline{n}) = \simp^{n}$, ${\Cubical}(\partial^{\epsilon}_{i})=\partial^{\epsilon}_{i}\vert_{\simp^{n}} : \simp^{n} \to \simp^{n+1}$ and ${\Cubical}(\varepsilon_{i})=\varepsilon_{i}\vert_{\simp^{n+1}} : \simp^{n+1} \to \simp^{n}$.
\begin{rem}
There is a smooth relative homeomorphism $\pi_{n} : (\Box^{n},\partial{\Box^{n}}) \to (\triangle^{n},\partial{\triangle^{n}})$ given by
$
\pi_{n}(t_{1},\cdots,t_{n}) = (s_{1},\ldots,s_{n}), \  s_{k} = t_{k} {\cdot}\cdots{\cdot} t_{n},
$
where the standard $n$-simplex $\triangle^{n}$ is assumed to be as 
$\triangle^{n} = \{(x_{1},\cdots,x_{n}) \in \real^{n} \,;\, 0=x_{0} \!\leq\! x_{1} \!\leq\! \cdots \!\leq\! x_{n} \!\leq\! x_{n+1}=1\}.$
\end{rem}

According to \cite{BH:11}, there is a natural embedding $ch: \diffeological \to \differentiable$.
So, from now on, we deal mainly with differentiable spaces, rather than diffeological spaces.
We denote $\mathcal{E}_{\Cubical}^{X}=\mathcal{E}^{X}{\circ}{\Cubical}$ and $\wedge^{p}_{\Cubical}=\wedge^{p}{\circ}{\Cubical}$, and a plot in $\mathcal{E}_{\Cubical}^{X}(\underline{n}) = \mathcal{E}^{X}(\simp^{n})$ is called an $n$-plot.

Let $X=(X,\mathcal{E}^{X})$ be a differentiable space. 
Then we denote $\Sigma_{n}(X) = \mathcal{E}^{X}(\simp^{n})$ the set of $n$-plots.
Let $\Gamma_{n}(X)$ be the free abelian group generated by $\Sigma_{n}(X)$ and $\Gamma^{n}(X,R) = \Hom(\Gamma_{n}(X);R)$, where $R$ is a commutative ring with unit.
Then $\Gamma^{\ast}(X;R)$ is a cochain complex and we obtain a smooth version of {\cubicaltxt} singular cohomology $H^{\ast}(X,R)$ in a canonical manner, which satisfies axioms of cohomology theories such as additivity, dimension and homotopy axioms together with a Mayer-Vietoris exact sequence.

\section{{\cubicaltxt} de Rham cohomology}

We introduce a version of a differential form by using $\mathcal{E}_{\Cubical}^{X}$ and $\wedge^{p}_{\Cubical}$.

\begin{defn}[{\cubicaltxt} differential form]
A {\cubicaltxt} differential form on a differentiable space $X$ is a natural transformation $\omega : \mathcal{E}^{X}_{\Cubical} \to \wedge^{p}_{\Cubical}$ of contravariant functors : ${\Cubical} \to \sets$. 
We denote $\omega = \{\omega_{\underline{n}}\,;\,n \!\geqq\! 0\}$, where $\omega_{\underline{n}} : \mathcal{E}^{X}(\simp^{n}) \to \wedge^{p}({\simp}^{n})$.
The set of {\cubicaltxt} differential forms on $X$ is denoted by $\Omega_{\Cubical}^{p}(X)$ and $\Omega^{\ast}_{\Cubical}(X) = \underset{p}{\oplus}\, \Omega_{\Cubical}^{p}(X)$.
\end{defn}

We denote by $\Cube^{\ast} : \Omega_{\mathcal{C}}^{p}(X){} \to \Omega_{\Cubical}^{p}(X)$ the natural map induced from $\Cube : {\Cubical} \rightarrow {\convex}$.
\begin{thm}
The map $\Cube^{\ast} : \Omega_{\mathcal{C}}^{p}(X){} \to \Omega_{\Cubical}^{p}(X)$ is monic.
\end{thm}
\begin{proof}
Assume that $\omega \in \Omega_{\mathcal{C}}^{p}(X)$ satisfies
$\Cube^{\ast}(\omega)=0 : \mathcal{E}^{X}_{\Cubical} \to \wedge^{p}_{\Cubical}$. 

By induction on $n$, we show $\omega_{A}=0$ for any convex $n$-domain $A$.
\par($n=0$) \ 
In this case, we have $\Omega_{\mathcal{C}}^{0}(X)=\Omega^{0}_{\Cubical}(X)$ and $\omega_{points}=0$.
\par($n>0$) \ 
Let $P : A \to X$ be a plot of $X$, where $A$ is a convex $n$-domain.
For any element $u \in \Int{A}$, there is a small {simplex} $\simp^{n}_{u} \subset \Int{A}$ such that $\Int{\simp^{n}_{u}} \ni u$.
Then there is a linear diffeomorphism $\phi : \simp^{n} \approx \simp^{n}_{u}$.
Hence $P{\circ}\phi \in C_{\mathcal{C}}^{\infty}(\simp^{n},X)$ and we obtain 
\par\vskip1ex\hfil
$0=\Cube^{\ast}(\omega)_{n}(P{\circ}\phi) =  \omega_{\simp^{n}}(P{\circ}\phi) = \phi^{*}(\omega_{\simp^{n}_{u}}(P|_{\simp^{n}_{u}})) = \phi^{*}(\omega_{A}(P)|_{\simp^{n}_{u}})$.
\hfil\par\vskip1ex\noindent
Since $\phi$ is a diffeomorphism, we have $\omega_{A}(P)|_{\simp_{u}^{n}}=0$ for any $u \in \Int{A}$.
Thus we obtain $\omega_{A}(P)=0$ on $\Int{A}$.
Since $\omega_{A}(P)$ is continuous, $\omega_{A}(P)=0$ on $A$. 
\end{proof}

A differentiable map induces a homomorphism of {\cubicaltxt} differential forms as follows:
\begin{defn}\label{defn:cubical-induced-map}
Let $f : X \to Y$ be a differentiable map between differentiable spaces $X = (X,\mathcal{E}^{X})$ and $Y = (Y,\mathcal{E}^{Y})$.
\begin{enumerate}
\item
We obtain a homomorphism $f^{\sharp} : \Omega^{p}_{\Cubical}(Y) \to \Omega^{p}_{\Cubical}(X)$: let $\omega \in \Omega_{\Cubical}^{p}(Y)$. Then 
\par\vskip.5ex\noindent
\hfil$
(f^{\sharp}\omega_{\underline{n}})(P)=\omega_{\underline{n}}(f{\circ}P) \quad \text{for any \ $P \in \mathcal{E}_{\Cubical}^{X}(\underline{n})$, $n \geq 0$.}
$\hfil
\vspace{.5ex}
\item
If a differentiable map $f$ is proper, then we have $f^{\sharp}(\Omega^{p}_{{\Cubical}_{c}}(Y)) \subset \Omega^{p}_{{\Cubical}_{c}}(X)$ by taking $K_{f^{\sharp}\omega}=f^{-1}(K_{\omega})$ for any $\omega \in \Omega^{p}_{{\Cubical}_{c}}(Y)$.
\end{enumerate}
\end{defn}

\begin{defn}[External derivative]
Let $X=(X,\mathcal{E})$ be a differentiable space.
The external derivative $\diff : \Omega_{\Cubical}^{p}(X) \to \Omega^{p+1}_{\Cubical}(X)$ is defined as follows.
\par\vskip1ex\noindent
\hfil$
(\diff\omega)_{\underline{n}}(P) = \diff(\omega_{\underline{n}}(P))
\quad\text{for an $n$-plot $P \in \mathcal{E}_{\Cubical}(\underline{n})=\mathcal{E}({\simp}^{n})$.}
$\hfil
\end{defn}

\begin{defn}
Let $X=(X,\mathcal{E})$ be  a differentiable space.
\begin{description}
\item[{\Cubicaltxt} de Rham cohomology]
$H_{\Cubical}^{p}(X) = \displaystyle\frac{Z_{\Cubical}^{p}(X)}{B_{\Cubical}^{p}(X)}$,
\par\vskip.5ex\noindent
where $Z_{\Cubical}^{p}(X) = \Ker{d}\cap \Omega_{\Cubical}^{p}(X)$
and $B_{\Cubical}^{p}(X) = \diff (\Omega_{\Cubical}^{p}(X))$.
\vskip.5ex
\item[{\Cubicaltxt} de Rham cohomology with compact support]
$H_{{\Cubical}_{c}}^{p}(X) = \displaystyle\frac{Z_{{\Cubical}_{c}}^{p}(X)}{B_{{\Cubical}_{c}}^{p}(X)}$,
\par\vskip.5ex\noindent
where $Z_{{\Cubical}_{c}}^{p}(X) = \Ker{d}\cap \Omega_{{\Cubical}_{c}}^{p}(X)$
and $B_{{\Cubical}_{c}}^{p}(X) = \diff (\Omega_{{\Cubical}_{c}}^{p}(X))$.
\end{description}
\end{defn}

\begin{expl}\label{expl:one-point-set-cube}
Let $X=(X,\mathcal{E}^{X})$ be a differentiable space with $X=\{\ast\}$ one-point-set.
Then we have $H_{\Cubical}^{p}(\{\ast\}) = \real$ if $p=0$ and $0$ otherwise.
\end{expl}

\begin{prop}
Let $X=(X,\mathcal{E}^{X})$ and $Y=(Y,\mathcal{E}^{Y})$ be differentiable spaces.
\begin{enumerate}
\item
For a differentiable map $f : X \to Y$, the homomorphism $f^{\sharp} : \Omega_{\Cubical}^{\ast}(Y) \to \Omega_{\Cubical}^{\ast}(X)$ induces a homomorphism $H_{\Cubical}^{\ast}(Y) \to H_{\Cubical}^{\ast}(X)$.
\item
If a differentiable map $f : X \to Y$ is proper, then the homomorphism $f^{\sharp} : \Omega_{{\Cubical}_{c}}^{\ast}(Y) \to \Omega_{{\Cubical}_{c}}^{\ast}(X)$ induces a homomorphism $f^{*} : H^{\ast}_{{\Cubical}_{c}}(Y) \to H^{\ast}_{{\Cubical}_{c}}(X)$.
\end{enumerate}
\end{prop}

\begin{thm}
By definition, we clearly have $H^{\ast}_{\Cubical}(\underset{\alpha}{\coprod}X_{\alpha}) = \underset{\alpha}{\prod}H^{\ast}_{\Cubical}(X_{\alpha})$.
\end{thm}

\begin{thm}
$H^{\ast}_{\Cubical}$ is a contravariant functor from $\differentiable$ 
to $\gradedalgebra$. 
\end{thm}

\section{Homotopy invariance of {\cubicaltxt} de Rham cohomology}

Let $f_{0}, f_{1} : X \to Y$ be homotopic differentiable maps between differentiable spaces $X=(X,\mathcal{E}^{X})$ and $Y=(Y,\mathcal{E}^{Y})$.
Then there is a plot $f : I \to C_{\mathcal{C}}^{\infty}(X,Y)$ with $f(t)=f_{t}$ for $t=0, 1$.
In particular, for any $n$-plot $P : \simp^{n} \to X$, 
$f{\cdot}P : \simp^{n+1} = I{\times}\simp^{n} \xrightarrow{f{\cdot}P} Y$ is an $n{+}1$-plot.
Then, we obtain a homomorphism $D_{f} : \Omega^{p}_{\Cubical}(Y) \to \Omega^{p-1}_{\Cubical}(X)$ as follows:
for any {\cubicaltxt} differential $p$-form $\omega : \mathcal{E}^{Y}_{\Cubical} \to \wedge^{p}_{\Cubical}$ on $Y$, a $p{-}1$-form $D_{f}(\omega) : \mathcal{E}^{X}_{\Cubical} \to \wedge^{p-1}_{\Cubical}$ on $X$ is defined by the following formula.
\par\vskip1ex\noindent\hfil$
\begin{array}{l}\displaystyle
D_{f}(\omega)_{\underline{n}}(P) =\int_{I}\omega_{\underline{n+1}}(f{\cdot}P) : \simp^{n} \to \wedge^{p-1}(T^{*}_{n}),
\\[2ex]
\displaystyle \left[\int_{I}\omega_{\underline{n+1}}(f{\cdot}P)\right](\mathbold{x}) = \underset{i_{2},\cdots,i_{p}}{\textstyle\sum}\int_{0}^{1}a_{i_{2},\cdots,i_{p}}(t,\mathbold{x})\diff{t}{\cdot}\diff{x_{i_{2}}}\wedge\cdots\wedge\diff{x_{i_{p}}},
\end{array}
$\hfil\par\vskip1ex\noindent
where we assume \vspace{-.5ex}$\omega_{\underline{n+1}}(f{\cdot}P)$ $=$ $\underset{i_{2},\cdots,i_{p}}{\textstyle\sum}a_{i_{2},\cdots,i_{p}}(t,\mathbold{x})\diff{t}\wedge\diff{x_{i_{2}}}\wedge\cdots\wedge\diff{x_{i_{p}}} + \underset{i_{1},\cdots,i_{p}}{\textstyle\sum}b_{i_{1},\cdots,i_{p}}(t,\mathbold{x})$ $\diff{x_{i_{1}}}\wedge\cdots\wedge\diff{x_{i_{p}}}$, $(t,\mathbold{x}) \in I{\times}\simp^{n} = \simp^{n+1}$ and $T^{*}_{n+1}=\real\diff{t}{\oplus}\overset{n}{\underset{i=1}{\oplus}}\real\diff{x_{i}}$.\vspace{-.5ex}
\begin{lem}\label{lem:homotopy-invariance}
For any $\omega$, we obtain $\diff{D(\omega)_{\underline{n}}} + D(\diff\omega)_{\underline{n}} = f^{\sharp}_{1}\omega_{\underline{n}}-f^{\sharp}_{0}\omega_{\underline{n}}$.
Thus, if $\diff\omega=0$, then $f^{\sharp}_{0}\omega$ is cohomologous to $f^{\sharp}_{1}\omega$.
\end{lem}
\begin{proof}
First, let $\omega_{\underline{n+1}}(f{\cdot}P) = \underset{i_{2},\cdots,i_{p}}{\textstyle\sum}a_{i_{2},\cdots,i_{p}}(t,\mathbold{x})\diff{t}\wedge\diff{x_{i_{2}}}\wedge\cdots\wedge\diff{x_{i_{p}}} + \underset{i_{1},\cdots,i_{p}}{\textstyle\sum}b_{i_{1},\cdots,i_{p}}(t,\mathbold{x})$ $\diff{x_{i_{1}}}\wedge\cdots\wedge\diff{x_{i_{p}}}$.
Let $\incl_{t} : \simp^{n} \to I{\times}\simp^{n}$ be the inclusion defined by $\incl_{t}(\mathbold{x})=(t,\mathbold{x})$ for $t=0, 1$.
Since $(f{\cdot}P){\circ}\incl_{t} = f_{t}{\circ}P$ for $t=0, 1$, we have $(f^{\sharp}_{t}\omega_{\underline{n}})(P) = \omega_{\underline{n}}(f_{t}{\circ}P) = \omega_{\underline{n}}((f{\cdot}P){\circ}\incl_{t})$ $=$ $\incl_{t}^{\ast}\,\omega_{\underline{n+1}}(f{\cdot}P)$ $=$ $\underset{i_{1},\cdots,i_{p}}{\textstyle\sum}b_{i_{1},\cdots,i_{p}}(t,\mathbold{x}) \diff{x_{i_{1}}}\wedge\cdots\wedge\diff{x_{i_{p}}}$ for $t=0, 1$, $\mathbold{x} \in \simp^{n}$.

Second, by definition, we have
$\diff\omega_{\underline{n+1}}(f{\cdot}P)$ $=$ $\displaystyle\underset{i}{\textstyle\sum}\underset{i_{2},\cdots,i_{p}}{\textstyle\sum}\pder{a_{i_{2},\cdots,i_{p}}}by{x_{i}}(t,\mathbold{x})$ $\diff{x_{i}}\wedge\diff{t}\wedge\diff{x_{i_{2}}}\wedge\cdots\wedge\diff{x_{i_{p}}} + \!\!\!\underset{i_{1},\cdots,i_{p}}{\textstyle\sum}\displaystyle\pder{b_{i_{1},\cdots,i_{p}}}by{t}(t,\mathbold{x}) \displaystyle\diff{t}\wedge\diff{x_{i_{1}}}\wedge\cdots\wedge\diff{x_{i_{p}}} + \underset{i}{\textstyle\sum}\underset{i_{1},\cdots,i_{p}}{\textstyle\sum}\displaystyle\pder{b_{i_{1},\cdots,i_{p}}}by{x_{i}}(t,\mathbold{x}) \diff{x_{i}}\wedge\diff{x_{i_{1}}}\wedge\cdots\wedge\diff{x_{i_{p}}}$, 
and hence we obtain 
$D(\diff\omega)_{\underline{n}}(P) = -\displaystyle\underset{i}{\textstyle\sum}\underset{i_{2},\cdots,i_{p}}{\textstyle\sum}\int_{I}\pder{a_{i_{2},\cdots,i_{p}}}by{x_{i}}(t,\mathbold{x})\diff{t}{\cdot}\diff{x_{i}}\wedge\diff{x_{i_{2}}}\wedge\cdots\wedge\diff{x_{i_{p}}} + \underset{i_{1},\cdots,i_{p}}{\textstyle\sum}\displaystyle\int_{I}\pder{b_{i_{1},\cdots,i_{p}}}by{t}(t,\mathbold{x})\diff{t}{\cdot}\diff{x_{i_{1}}}\wedge\cdots\wedge\diff{x_{i_{p}}}$, $(t,\mathbold{x}) \in I{\times}\simp^{n}$.

Third, we have $D_{f}(\omega)_{\underline{n}}(P) = \displaystyle\underset{i_{2},\cdots,i_{p}}{\textstyle\sum}\int_{I}a_{i_{2},\cdots,i_{p}}(t,\mathbold{x})\diff{t}{\cdot}\diff{x_{i_{2}}}\wedge\cdots\wedge\diff{x_{i_{p}}}$, and hence we obtain 
$\diff{D_{f}(\omega)_{\underline{n}}}(P) = \displaystyle\underset{i}{\textstyle\sum}\underset{i_{2},\cdots,i_{p}}{\textstyle\sum}\int_{I}\pder{a_{i_{2},\cdots,i_{p}}}by{x_{i}}(t,\mathbold{x})\diff{t}{\cdot}\diff{x_{i}}\wedge\diff{x_{i_{2}}}\wedge\cdots\wedge\diff{x_{i_{p}}}$, $(t,\mathbold{x}) \in I{\times}\simp^{n}$.

Hence $\left[\diff{D_{f}(\omega)_{\underline{n}}}(P) + D_{f}(\diff\omega)_{\underline{n}}(P)\right](\mathbold{x}) = \underset{i_{1},\cdots,i_{p}}{\textstyle\sum}\displaystyle\int_{I}\pder{b_{i_{1},\cdots,i_{p}}}by{t}(t,\mathbold{x})\diff{t}{\cdot}\diff{x_{i_{1}}}\wedge\cdots\wedge\diff{x_{i_{p}}} = \underset{i_{1},\cdots,i_{p}}{\textstyle\sum} b_{i_{1},\cdots,i_{p}}(1,\mathbold{x}) \diff{x_{i_{1}}}\wedge\cdots\wedge\diff{x_{i_{p}}} - \underset{i_{1},\cdots,i_{p}}{\textstyle\sum} b_{i_{1},\cdots,i_{p}}(0,\mathbold{x}) \diff{x_{i_{1}}}\wedge\cdots\wedge\diff{x_{i_{p}}}$, $\mathbold{x} \in \simp^{n}$.
Thus we obtain $\diff{D_{f}(\omega)_{\underline{n}}}(P) + D_{f}(\diff\omega)_{\underline{n}}(P) = (f^{\sharp}_{1}\omega_{\underline{n}})(P) - (f^{\sharp}_{0}\omega_{\underline{n}})(P)$, which implies the lemma.
\end{proof}

It immediately implies the following theorem.

\begin{thm}\label{thm:homotopy-invariance}
If two differentiable maps $f_{0}, \ f_{1} : X \to Y$ between differentiable spaces are homotopic in $C_{\mathcal{C}}^{\infty}(X,Y)$, then they induce the same homomorphism 
\par\noindent\hfil
$f^{*}_{0} = f^{*}_{1} : H_{\Cubical}^{\ast}(Y) \to H_{\Cubical}^{\ast}(X)$.
\end{thm}

\section{Hurewicz homomorphism}

First, we give a definition of paths and fundamental groupoid of a differentiable space.

\begin{defn}
In this paper, a path from $a \in X$ to $b \in X$ in a differentiable space $X$ means a differentiable map $\ell : I \to X$ such that $\ell(0)=a$ and $\ell(1)=b$.
We denote by $\pi_{0}(X)$ the set of path-connected components of $X$, as usual.
\end{defn}

\begin{defn}
Let $\cats$ be the category of all small categories.
The fundamental groupoid functor $\underline{\pi}_{1} : \differentiable \to \cats$ is as follows: 
\begin{enumerate}
\item
For a differentiable space $X$, the small category $\underline{\pi}_{1}(X)$ is defined by
$\Object{\underline{\pi}_{1}(X)} = X$ \ and \ 
$\Morphism{\underline{\pi}_{1}(X)}(x_{0},x_{1})$ is the set of homotopy classes of all differentiable maps $\ell : I \to X$ with $\ell(0)=x_{0}$ and $\ell(1)=x_{1}$ for any $x_{0}, x_{1} \in X$.
\vspace{.5ex}
\item
For a differentitable map $f : Y \to X$, the functor $f_{*} : \underline{\pi}_{1}(Y) \to \underline{\pi}_{1}(X)$ is defined by
$f_{*} = f : Y \to X$ \ and \ 
$f_{*}([\ell]) = [f{\circ}\ell]$ for any $[\ell] \in \underline{\pi}_{1}(Y)$.
\end{enumerate}
\end{defn}

\begin{defn}
The functor $\underline{\real} : \differentiable \to \cats$ is defined as follows:
\begin{enumerate}
\item
For a differentiable space $X$, the small category $\underline{\real}(X)$ is defined by
$\Object{\underline{\real}(X)} = X$ \ and \ 
$\Morphism{\underline{\real}(X)}(x_{0},x_{1}) = \real$\, for any $x_{0}, x_{1} \in X$, and the composition is given by addition of real numbers.
\vspace{.5ex}
\item
For a differentitable map $f : Y \to X$, the functor $f_{*} : \underline{\real}(Y) \to \underline{\real}(X)$ is defined by
$f_{*} = f : Y \to X$ \ and \ 
$f_{*} = \id : \real \to \real$.
\end{enumerate}
\end{defn}

\begin{defn}\label{defn:hurewicz_homomorphism}
The Hurewicz homomorphism $\rho : Z_{\Cubical}^{1}(X) \to \Hom(\underline{\pi}_{1}(X),\underline{\real}(X))$ (the set of functors) is defined for any $\omega \in Z_{\Cubical}^{1}(X)$ by \vspace{1ex}
$\rho(\omega)(x)=x$ for any $x \in \Object{\underline{\pi}_{1}(X)}=X$ and 
$\displaystyle\rho(\omega)([\ell]) = \int_{I}\omega_{\underline{1}}(\ell)$ for any $[\ell] \in \Morphism{\underline{\pi}_{1}(X)}$,\vspace{1ex}
which is natural, in other words, the diagram below is commutative for any differentiable map $f : Y \to X$ between differentiable spaces. 
\par\vskip1ex\noindent
\hfil$
\begin{diagram}
\node{Z_{\Cubical}^{1}(X)}
\arrow{e,t}{\rho}
\arrow{s,l}{f^{*}}
\node{\Hom(\underline{\pi}_{1}(X),\underline{\real}(X))}
\arrow{s,r}{\Hom(f_{*},\id)}
\\
\node{Z_{\Cubical}^{1}(Y)}
\arrow{e,t}{\rho}
\node{\Hom(\underline{\pi}_{1}(Y),\underline{\real}(Y))}
\end{diagram}
$\hfil
\par\vskip1ex\noindent
(well-defined)
Let $\ell_{0} \sim \ell_{1}$ with $\ell_{t}(\epsilon)=x_{\epsilon} \in X$, $t=0,1$ and $\epsilon=0,1$.
Then there is a $2$-plot $\hat\ell : \simp^{2} \to X$ such that $\hat\ell(\epsilon,s)=\ell_{\epsilon}(s)$ and $\hat\ell(t,\epsilon)=x_{\epsilon}$ for $\epsilon=0,1$.
Hence we have $\hat\ell{\circ}\partial^{\epsilon}_{1}=\ell_{\epsilon}$, $\epsilon=0,1$ and $\hat\ell{\circ}\partial^{\epsilon}_{2}=c_{x_{\epsilon}}=c_{x_{\epsilon}}{\circ}\varepsilon_{1}$.
Let $\omega_{\underline{2}}(\hat\ell) = a(t,s)\diff{t} + b(t,s)\diff{s} \in \wedge^{1}(\simp^{2})$.
Then we have $\omega_{\underline{2}}(\ell_{\epsilon})=\omega_{\underline{2}}(\hat\ell{\circ}\partial^{\epsilon}_{1})=\partial_{1}^{\epsilon*}\omega_{\underline{2}}(\hat\ell)=b(\epsilon,s)\diff{s}$, $\epsilon=0,1$.
Similarly, $0=\varepsilon_{1}^{*}\omega_{\ast}(c_{x_{\epsilon}})=\omega_{\underline{2}}(c_{x_{\epsilon}}{\circ}\varepsilon_{1})=\omega_{\underline{2}}(\hat\ell{\circ}\partial^{\epsilon}_{2})=\partial_{2}^{\epsilon*}\omega_{\underline{2}}(\hat\ell)=a(t,\epsilon)\diff{t}$ which implies $a(t,\epsilon)=0$, $\epsilon=0,1$.
On the other hand by Green's formula, \vspace{1ex}we obtain that $\displaystyle\int_{\partial\,\simp^{2}}(\omega_{\underline{2}}(\hat\ell)\vert_{\partial\,\simp^{2}})=\int_{\simp^{2}}\diff\omega=0$, since $\omega$ is a closed form.\vspace{.5ex}
Then it follows that $\displaystyle\int_{\{1\}{\times}I}(\omega_{\underline{2}}(\hat\ell)\vert_{\{1\}{\times}I})-\displaystyle\int_{\{0\}{\times}I}(\omega_{\underline{2}}(\hat\ell)\vert_{\{0\}{\times}I})=0$, and hence \vspace{1ex}$\displaystyle\int_{I}\omega_{\underline{1}}(\ell_{1})=\int_{I}\omega_{\underline{1}}(\ell_{0})$, and $\rho$ is well-defined.
The additivity of $\rho$ is clear by definition.
\par(naturality)
Let $f : Y \to X$ be a differentiable map.
Then $f$ induces both $f^{*} : Z_{\Cubical}^{1}(X) \to Z_{\Cubical}^{1}(Y)$ and $f_{*} : \underline{\pi}_{1}(Y) \to \underline{\pi}_{1}(X)$.
The latter homomorphism induces 
\par\vskip1ex\noindent\hfil
$\Hom(f_{*},\id) : \Hom(\underline{\pi}_{1}(X),\underline{\real}(X)) \to \Hom(\underline{\pi}_{1}(Y),\underline{\real}(Y))$.\hfil
\par\vskip1ex\noindent
Then, for any $\omega \in Z_{\Cubical}^{1}(X)$ and $[\ell] \in \underline{\pi}_{1}(X)$, it follows that 
\par\vskip1ex\noindent\hfil
$\displaystyle\rho(f^{*}(\omega))([\ell]) = \int_{I}(f^{\sharp}\omega_{\underline{1}})(\ell)=\int_{I}\omega_{\underline{1}}(f{\circ}\ell) 
= \rho(\omega)([f{\circ}\ell]) = \rho(\omega){\circ}f_{*}([\ell])$
\hfil\par\vskip1ex\noindent
and hence we have $\rho{\circ}f^{*}=\Hom(f_{*},\id){\circ}\rho$ which implies the naturality of $\rho$.
\end{defn}

\begin{defn}
For any differentiable space $X$, we define a groupoid $\underline{X}$ in which the set of objects is equal to $X = \Object{\underline{\pi}_{1}(X)}$, and the set of morphisms is obtained from $\Morphism{\underline{\pi}_{1}(X)}$ by identifying all the morphisms which have starting and ending objects in common.
\end{defn}

Then there clearly is a natural projection $\proj : \underline{\pi}_{1}(X) \to \underline{X}$ inducing a monomorphism $\proj^{*} : \Hom(\underline{X},\underline{\real}(X)) \hookrightarrow \Hom(\underline{\pi}_{1}(X),\underline{\real}(X))$.
\begin{defn}
We denote the cokernel of $\proj^{*}$ by $\Hom(\underline{\pi}_{1}(X),\real)$.
\end{defn}

If $\omega=\diff{\phi}$ for some $\phi \in \Omega_{\Cubical}^{0}(X)$, then, for any path $\ell$ from $x_{0}$ to $x_{1}$, \vspace{.2ex}we have $\rho(\omega)([\ell])=\rho(\diff{\phi})([\ell]) = \displaystyle\int_{I}\diff{(\phi_{I})}(\ell) = [\phi_{I}(\ell)(t)]_{t=0}^{t=1}=\phi_{I}(\ell)(1)-\phi_{I}(\ell)(0)$, \vspace{.5ex}by the fundamental theorem of calculus.
Hence $\phi_{I}(\ell)(\epsilon)=\phi_{I}(\ell)(\partial^{\epsilon}_{1}(\ast)) = \partial_{1}^{\epsilon\ast}(\phi_{I}(\ell))(\ast) = \phi_{\{\ast\}}(\ell{\circ}\partial^{\epsilon}_{1})(\ast) = \phi_{\{\ast\}}(\ell(\epsilon))(\ast) = \phi_{\{\ast\}}(c_{x_{\epsilon}})(\ast)$ is depending only on $x_{\epsilon}$ the starting and ending objects of $[\ell] \in \underline{\pi}_{1}(X)$. 
Thus the functor $\rho(\omega) : \underline{\pi}_{1}(X) \to \underline{\real}(X)$ induces a functor $\Phi(\omega) : \underline{X} \to \underline{\real}(X)$ such that $\rho(\omega)=\Phi(\omega){\circ}\proj$, in other words, $\rho(B^{1}_{\Cubical}(X))$ is in the image of $\proj^{*}$.
Thus $\rho$ induces a homomorphism $\rho_{*} : H^{\ast}_{\Cubical}(X) \to \Hom(\underline{\pi}_{1}(X),\real)$.

\begin{thm}\label{thm:hurewicz-mono}
$\rho_{*} : H^{1}_{\Cubical}(X) \to \Hom(\underline{\pi}_{1}(X),\real)$ is a monomorphism.
\end{thm}
\begin{proof}
Assume that $\rho_{*}([\omega])=0$.
Then we have $\rho(\omega) \in \Im{\proj^{*}}$.
Thus there is a functor $\Phi(\omega) : \underline{X} \to \underline{\real}$ such that $\rho(\omega)=\Phi(\omega){\circ}\proj$.
Let $\{x_{\alpha}\,;\,\alpha\!\in\!\pi_{0}(X)\}$ be a complete set of representatives of $\pi_{0}(X)$.
For any $P \in \mathcal{E}(\simp^{n})$, a map $F(P) : \simp^{n} \to \real$ is given by 
\par\vskip.7ex\noindent\hfil
$F(P)(\mathbold{x})$ $=$ $\displaystyle \int_{I}\omega_{\underline{1}}(\ell_{x})+ \int_{I}\gamma^{*}_{\mathbold{x}}\omega_{\underline{n}}(\simp^{n})$, $x\!=\!P(\mathbold{0})$, 
\hfil\par\vskip1ex\noindent
where $\ell_{x}$ is a path from $x_{\alpha}$, $\alpha\!=\![x]\!\in\!\pi_{0}(X)$, \vspace{.5ex}to $x$ in $X$ and $\gamma$ is a path from $\mathbold{0}$ to $\mathbold{x}$ in $\simp^{n}$.
Then $F(P) : \simp^{n} \to \wedge^{0}$ is well-defined smooth map by the equality \vspace{.5ex}$\displaystyle \int_{I}\omega_{\underline{1}}(\ell_{x}) = \rho(\omega)([\ell_{x}]) = \Phi(\omega)(\proj([\ell_{x}]))$ \vspace{.5ex}which is not depending on the choice of $\ell_{x}$, and hence it gives a $0$-form $F : \mathcal{E}(\simp^{n}) \to \wedge^{0}(\simp^{n})$ so that $\diff{F}=\omega$.
Thus $[\omega]=0$ and $\rho_{*}$ is a monomorphism.
\end{proof}

\section{Partition of unity}

Let $X$ be a differentiable space.
In this section, we assume that there are subsets $A, B \subset X$ such that \,$\mathcal{U} = \{\Int{A},\Int{B}\}$ gives an open covering of $X$.

\begin{defn}
A pair $(\rho^{A},\rho^{B})$ of differentiable $0$-forms $\rho^{A}$ and $\rho^{B}$ is called a partition of unity belonging to an open covering \,$\mathcal{U}$ of $X$, if, for any plot $P : \simp^{n} \to X$, $\Supp{\rho_{\underline{n}}^{A}(P)} \subset P^{-1}(\Int{A})$, $\Supp{\rho_{\underline{n}}^{B}(P)} \subset P^{-1}(\Int{B})$ and $\rho_{\underline{n}}^{A}(P)+\rho_{\underline{n}}^{B}(P)=1$ on $\simp^{n}$.
\end{defn}

To obtain a well-defined smooth function by extending or gluing smooth functions on {\cubictxt}s, we use a fixed smooth stabilizer function $\hat\lambda : \real \to I$ (see \cite{Iglesias:13}) which satisfies 
\begin{enumerate}
\item
$\hat\lambda(-t)=0$, $\hat\lambda(1{+}t)=1$, $t \geq 0$ \quad and 
\hitem
$\hat\lambda$ is strictly increasing on $I=[0,1]$.
\end{enumerate}
Using $\hat\lambda$, we define a smooth function $\lambda_{a,b} : I \to I$, for any $a,b \in \real$ with $a\!<\!b$, by $$\textstyle\lambda_{a,b}(t) = \hat\lambda(\frac{t{-}a{-}\epsilon}{b{-}a{-}2\epsilon})$$ for a small $\epsilon\!>\!0$ enough to satisfy $\frac{b{-}a}{2}\!>\!\epsilon\!>\!0$.

Using it, we show the existence of a partition of unity as follows.

\begin{thm}\label{thm:partition_of_unity}
Let $X$ be a differentiable space with an open covering $\{\Int{A},\Int{B}\}$, $A, B \subset X$.
Then there exists a partition of unity $\mathbold\rho=\{\rho^{A},\rho^{B}\}$ belonging to $\{\Int{A},\Int{B}\}$.
If the underlying topology on $X$ is normal, $\mathbold\rho$ can be chosen as normal, in other words, there are closed sets $G_{A}$, $G_{B}$ in $X$ such that $X\!\smallsetminus\!\Int{B} \subset G_{A} \subset \Int{A}$, $X\!\smallsetminus\!\Int{A} \subset G_{B} \subset \Int{B}$ and $\Supp{\rho^{A}_{\underline{n}}(P)} \subset P^{-1}(G_{A})$ and $\Supp{\rho^{B}_{\underline{n}}(P)} \subset P^{-1}(G_{B})$ for all $n \!\geq\! 0$ and $P \in \mathcal{E}^{X}(\simp^{n})$.\vspace{-.5ex}
\end{thm}

The above theorem implies the exactness of Mayer-Vietoris exact sequence as follows.

\begin{cor}
Let $X$ be a differentiable space with an open covering \,$\mathcal{U}=\{\Int{A},\Int{B}\}$, $A, B \subset X$.
Then we have the following long exact sequence.
\par\vskip1.5ex\noindent\hfil$
\begin{array}{l}\displaystyle
\cdots \to H_{\Cubical}^{q}(X) \to H_{\Cubical}^{q}(A){\oplus}H_{\Cubical}^{q}(B) \to H_{\Cubical}^{q}(A \cap B) 
\\[1.5ex]\displaystyle\qquad\qquad
\to H_{\Cubical}^{q+1}(X) \to H_{\Cubical}^{q+1}(A){\oplus}H_{\Cubical}^{q+1}(B) \to H_{\Cubical}^{q+1}(A \cap B)  \cdots
\end{array}$\hfil%
\par\vskip1.5ex\noindent
\end{cor}
\begin{proof}[Theorem \ref{thm:partition_of_unity}]
If $X$ is normal, there is a continuous function $\rho : X \to I$ with $X \!\smallsetminus\! \Int{B} \subset \rho^{-1}(0)$ and $X \!\smallsetminus\! \Int{A} \subset \rho^{-1}(1)$.
Otherwise, we define a function $\rho : X \to I$ by
\par\vskip1ex\noindent\hfil
$\rho(x)=\begin{cases}
1, & x \in \Int{A} \smallsetminus \Int{B}, 
\\[.2ex] 1/2, & x \in \Int{A} \cap \Int{B}, 
\\[.2ex] 0 & x \in \Int{B} \smallsetminus \Int{A}.
\end{cases}$
\hfil\par\vskip1ex\noindent

Let $G_{A}=\rho^{-1}([0,\frac{2}{3}]) \subset X\!\smallsetminus\!\rho^{-1}(1) \subset \Int{A}$ and $G_{B}=\rho^{-1}([\frac{1}{3},1]) \subset X\!\smallsetminus\!\rho^{-1}(0) \subset \Int{B}$.
Then $\Int{G_{A}} \cup \Int{G_{B}} \supset \rho^{-1}([0,\frac{2}{3})) \cup \rho^{-1}((\frac{1}{3},1]) = \rho^{-1}([0,\frac{2}{3}) \cup (\frac{1}{3},1]) = X$.
Thus it is sufficient to construct a partition of unity $\{\rho^{A},\rho^{B}\}$ belonging to $\mathcal{U}=\{\Int{G_{A}},\Int{G_{B}}\}$:
by induction on $n$, we construct functions $\rho_{\underline{n}}^{A}(P),\, \rho_{\underline{n}}^{B}(P) : \simp^{n} \to I$ for any $n$-plot $P : \simp^{n} \to X$, with conditions (1) through (4) below for $F=A, B$ and $\epsilon=0, 1$.
\begin{enumerate}
\vskip1ex
\item\label{cold:partition-of-unity0}
\begin{enumerate}
\item
$\rho_{\underline{n}}^{F}(P{\circ}\varepsilon_{i})=\rho_{\underline{n-1}}^{F}(P){\circ}\varepsilon_{i}$, \ $1 \!\leq\! i \!\leq\! n{+}1$,
\hiitem
$\rho_{\underline{n-1}}^{F}(P{\circ}\partial^{\epsilon}_{i})=\rho_{\underline{n}}^{F}(P){\circ}\partial^{\epsilon}_{i}$, \ $1 \!\leq\! i \!\leq\! n$, 
\end{enumerate}
\vskip1ex
\item\label{cold:partition-of-unity1}
$\rho_{\underline{n}}^{A}(P)+\rho_{\underline{n}}^{B}(P)=1 : \simp^{n} \to \real$,
\hitem\label{cold:partition-of-unity2}
$\Supp\rho_{\underline{n}}^{F}(P) \subset {P}^{-1}(\Int{G_{F}}) \subset \simp^{n}$,
\vskip1ex
\item\label{cold:partition-of-unity4}
$\rho_{F}(P){\circ}\partial^{1-t}_{i} = \rho_{F}(P){\circ}\partial^{1}_{i}$ and 
$\rho_{F}(P){\circ}\partial^{t}_{i} = \rho_{F}(P){\circ}\partial^{0}_{i}$\vspace{.5ex} for all $0 \!\leq\! t \!\leq\! a$ for sufficiently small $a\!>\!0$, where $\partial^{t}_{i}$ is defined by $\partial^{t}_{i}(t_{1},\ldots,t_{n-1}) = (t_{1},\ldots,t_{i-1},t,t_{i+1},\ldots,t_{n-1})$.
\end{enumerate}
\par\vskip-.5ex\noindent
\par($n=0$) \ 
For any plot $P : \simp^{0} = \{\ast\} \to X$, we define $\rho_{\underline{n}}^{A}(P) = \rho(P(*))$ and $\rho_{\underline{n}}^{B}(P)=1-\rho_{\underline{n}}^{A}(P)$, which satisfy (2) and (3), though (1) and (4) are empty conditions in this case.
\par\vskip1ex\noindent
\par($n>0$) \ 
We may assume a plot $P : \simp^{n} \to X$ is non-degenerate by (1) a).
\par\vskip1ex\noindent

Firstly, $P^{-1}\mathcal{U}=\{P^{-1}(\Int{A}), P^{-1}(\Int{B})\}$ is an open covering of $\simp^{n} \subset \real^{n}$, and hence we have a partition of unity $\{\varphi^{A},\varphi^{B}\}$ belonging to $P^{-1}\mathcal{U}$ on $\simp^{n}$.

Secondly, by the induction hypothesis, there is a small $a\!>\!0$ for the condition (4).
Let $U_{a}$ be the $a$-neighbourhood of $\partial{\simp^{n}}$. 
For $F=A, B$, we define $\hat\rho_{\underline{n}}^{F}(P) : U_{a} \to \real$ by
\par\vskip1ex\noindent
\hfil$
\hat\rho_{\underline{n}}^{F}(P){\circ}\partial^{\,\epsilon{\pm}t}_{i} = 
\rho_{\underline{n-1}}^{F}(P{\circ}\partial^{\epsilon}_{i}), \ 0 \!\leq\! t \!<\! a, \ 1 \!\leq\! i \!\leq\! n, \ \epsilon\!=\!0, 1,
$\hfil
\par\vskip1ex\noindent
where we denote $\epsilon{\pm}t = \epsilon\!+\!(-1)^{\epsilon}t$, and then we obtain $\Supp\hat\rho_{\underline{n}}^{F}(P) \subset P^{-1}(\Int{G_{F}}) \cap U_{a}$, if we choose $a\!>\!0$ small enough.
\par\vskip1.5ex\noindent
\setlength\unitlength{.1truemm}
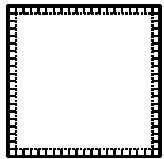
\begin{center}
\begin{picture}(400,200)(-200,-100)
\multiput(-90,-90)(5,0){36}{\line(1,0){2}}
\multiput(-90,90)(5,0){36}{\line(1,0){2}}
\multiput(-90,-87)(0,5){36}{\line(0,1){2}}
\multiput(90,-87)(0,5){36}{\line(0,1){2}}
\multiput(-100,-90)(0,10){19}{\line(1,0){10}}
\multiput(90,-90)(0,10){19}{\line(1,0){10}}
\multiput(-90,-100)(10,0){19}{\line(0,1){10}}
\multiput(-90,90)(10,0){19}{\line(0,1){10}}
\put(120,-100){$U_{a}$}
\thicklines
\multiput(-100,-100)(0,200){2}{\line(1,0){200}}
\multiput(-100,-100)(200,0){2}{\line(0,1){200}}
\end{picture}
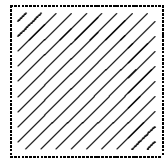
\begin{picture}(400,200)(-200,-100)
\multiput(-100,-100)(5,0){40}{\line(1,0){2}}
\multiput(-100,100)(5,0){40}{\line(1,0){2}}
\multiput(-100,-97)(0,5){40}{\line(0,1){2}}
\multiput(100,-97)(0,5){40}{\line(0,1){2}}
\multiput(-80,-90)(-10,10){2}{\line(1,1){170}}
\multiput(-60,-90)(-30,30){2}{\line(1,1){150}}
\multiput(-40,-90)(-50,50){2}{\line(1,1){130}}
\multiput(-20,-90)(-70,70){2}{\line(1,1){110}}
\multiput(0,-90)(-90,90){2}{\line(1,1){90}}
\multiput(20,-90)(-110,110){2}{\line(1,1){70}}
\multiput(40,-90)(-130,130){2}{\line(1,1){50}}
\multiput(60,-90)(-150,150){2}{\curve(0,0,30,30)}
\multiput(80,-90)(-170,170){2}{\curve(0,0,10,10)}
\put(120,-100){$\Int{\simp^{n}}$}
\end{picture}
\end{center}
\par\vskip1.5ex\noindent

Thirdly, since two open sets $U_{a}$ and $\Int{\simp^{n}}$ form an open covering of $\simp^{n}$, we also have a partition of unity $(\psi_{\partial},\psi_{\circ})$ belonging to $\{U_{a},\Int{\simp^{n}}\}$ given by $\psi_{\partial} = (\lambda_{1-a,1})^{n}
$ and $\psi_{\circ} = 1\!-\!\psi_{\partial}$ so that we have $\Supp\psi_{\partial} \subset U_{a}$ and $\Supp\psi_{\circ} \subset \Int{\simp^{n}}$.
Then, for $F=A, B$, $\psi_{\partial}\vert_{U_{a}}{\cdot}\hat\rho_{\underline{n}}^{F}(P)$ is defined on $U_{a}$ with value $0$ on $U_{a} \!\smallsetminus\! \Supp\psi_{\partial}$.
Hence by filling $0$ outside $\Supp\psi_{\partial}$, we obtain a smooth map $\widehat{\psi_{\partial}\rho_{\underline{n}}^{F}} : \simp^{n} \to \real$ on entire $\simp^{n}$, as the $0$-extension of $\psi_{\partial}\vert_{U_{a}}{\cdot}\hat\rho_{\underline{n}}^{F}(P) : U_{a} \to \real$.

Finally, let $\rho_{\underline{n}}^{F}(P)=\widehat{\psi_{\partial}\rho_{\underline{n}}^{F}}+\psi_{\circ}{\cdot}\varphi^{F}$ for $F=A, B$.
Then $\Supp\rho_{\underline{n}}^{F}(P) \subset \Supp\widehat{\psi_{\partial}\rho_{\underline{n}}^{F}} \cup \Supp(\psi_{\circ}{\cdot}\varphi^{F}) \subset (\Supp\psi_{\partial} \,\cap\, \Supp\hat\rho_{\underline{n}}^{F}) \,\cup\, (\Supp\psi_{\circ} \,\cap\, \Supp\varphi^{F}) \subset (U_{a} \cap P^{-1}(\Int{G_{F}})) \,\cup\, (\Int{\simp^{n}} \,\cap\, P^{-1}(\Int{G_{F}})) = P^{-1}(\Int{G_{F}})$.
By definition, we also have 
\par\vskip1ex\noindent\hfil
$\rho_{\underline{n}}^{A}(P)+\rho_{\underline{n}}^{B}(P)=\widehat{\psi_{\partial}\rho_{\underline{n}}^{A}}+\widehat{\psi_{\partial}\rho_{\underline{n}}^{B}}+\psi_{\circ}{\cdot}\varphi^{A}+\psi_{\circ}{\cdot}\varphi^{B} = \psi_{\partial} + \psi_{\circ}=1$ \ on \ $\simp^{n}$, 
\hfil\par\vskip1ex\noindent
which implies that $(\rho_{\underline{n}}^{A}(P),\rho_{\underline{n}}^{B}(P))$ gives a partition of unity belonging to the open covering $\{P^{-1}(\Int{A}),P^{-1}(\Int{B})\}$ of $\simp^{n}$.
By definition, $(\rho_{\underline{n}}^{A}(P),\rho_{\underline{n}}^{B}(P))$ satisfies the conditions (1) through (4), and it completes the induction step.
The latter part is clear.
\end{proof}

\section{Excision theorem}

Let $X=(X,\mathcal{E}^{X})$ be a differentiable space and \,$\mathcal{U}$ an open covering of $X$.
We denote $\mathcal{E}^{\,\mathcal{U}} = \{P \!\in\! \mathcal{E}^{X}\,;\,\Im{P} \!\subset\! U \ \text{for some} \ U \!\in\! \mathcal{U}\}$.
Then we regard $\mathcal{E}^{\,\mathcal{U}}$ as a functor $\mathcal{E}^{\,\mathcal{U}} : \convex \to \sets$ which is given by $\mathcal{E}^{\,\mathcal{U}}(C)=\{P \in \mathcal{E}^{\,\mathcal{U}}, \ \Dom{P}=C\}$ for $C \in \Object{\convex}$ and $\mathcal{E}^{\,\mathcal{U}}(f) = \mathcal{E}^{X}(f)\vert_{\mathcal{E}^{\,\mathcal{U}}(C)} : \mathcal{E}^{\,\mathcal{U}}(C) \to \mathcal{E}^{\,\mathcal{U}}(C')$ for a smooth map $f : C' \to C$ in $\convex$.
When \,$\mathcal{U}=\{\!X\!\}$, we have $\mathcal{E}^{\{\!X\!\}} = \mathcal{E}^{X}$.
We also denote $\mathcal{E}_{\Cubical}^{\,\mathcal{U}} = \mathcal{E}^{\,\mathcal{U}}{\circ}{\Cubical} : \Cubical \to \sets$.

\begin{defn}
A natural transformation $\omega : \mathcal{E}_{\Cubical}^{\,\mathcal{U}}\to \wedge^{p}_{\Cubical}$ is called a {\cubicaltxt} differencial $p$-form w.r.t. an open covering \,$\mathcal{U}$ of $X$.
$\Omega_{\Cubical}^{p}(\mathcal{U})$ denotes the set of all {\cubicaltxt} differential $p$-form w.r.t. an open covering \,$\mathcal{U}$ of $X$.
For example, $\Omega_{\Cubical}^{p}(\{\!X\!\})=\Omega_{\Cubical}^{p}(X)$.
\end{defn}

We introduce a notion of a $q$-{\cubictxt} in $\real^{n}$ using induction on $q \geq -1$ up to $n$.
\begin{description}
\item[~~($q\!=\!-1$)]
The empty set $\emptyset$ is a $-1$-cubic set in $\real^{n}$.
\vspace{.5ex}
\item[~~($n \!\geq\! q\!\geq\!0$)]
\begin{enumerate}
\vspace{.5ex}
\item
if $\sigma \subset L$ is a $(q{-}1)$-{\cubictxt} in $\real^{n}$ and $\mathbold{b} \!\not\in\! L$, where $L$ is a hyperplane of dimension $q{-}1$ in $\real^{n}$, then $\sigma{\ast}\mathbold{b} = \{t\mathbold{x}\!+\!(1{-}t)\mathbold{b} \,;\, \mathbold{x} \!\in\! \sigma, \ t \!\in\! I\}$ is a $q$-{\cubictxt} in $\real^{n}$ with faces $\tau$ and $\tau{\ast}\mathbold{b}$, where $\tau$ is a face of $\sigma$, including $\emptyset$ and $\emptyset{\ast}\mathbold{b}=\mathbold{b}$.
\vspace{.5ex}
\item
if $\sigma \subset \real^{i-1}{\times}\{0\}{\times}\real^{n-i}$ is a $(q{-}1)$-{\cubictxt} in $\real^{n}$ with $q \!\geq\! 1$, then the product set $\sigma{\times}I = \{(\mathbold{x}_{i-1},t,\mathbold{x}'_{n-i}) \,;\, (\mathbold{x}_{i-1},0,\mathbold{x}'_{n-i}) \!\in\! \sigma, \ t \!\in\! I\}$ is a $q$-{\cubictxt} in $\real^{n}$ with faces $\tau{\times}\{0\}$, $\tau{\times}\{1\}$ and $\tau{\times}I$, where $\tau$ is a face of $\sigma$, including $\emptyset$.
\vspace{1ex}
\end{enumerate}
\end{description}
We denote by $\Cubic(n)^{q}$ the set of $q$-{\cubicstxt} in $\real^{n}$ and $\Cubic(n) = \{\emptyset\}\cup\underset{q \geq 0}{\cup}\Cubic(n)^{q}$, $n \geq 0$.
We denote $\tau < \sigma$ if $\tau \in \Cubic(n)$ is a face of $\sigma \in \Cubic(n)$ and denote $\partial{\sigma} = \underset{\tau<\sigma}{\cup}\sigma$.
We fix a relative diffeomorphism $\phi_{\sigma} : ({\Cubical}^{q},\partial{{\Cubical}^{q}}) \twoheadrightarrow (\sigma,\partial{\sigma})$ for each $q$-{\cubictxt} $\sigma$ in $\real^{n}$, $q \!\geq\! 0$.

A subset $K \subset \Cubic(n)$ is called a {\cubicaltxt} complex if it satisfies the following conditions.
\begin{enumerate}
\vspace{.0ex}
\item
$\emptyset \in K$,
\hitem
$\tau<\sigma \ \& \ \sigma \in K \implies \tau \in K$,
\vspace{.5ex}
\item
$\tau, \sigma \in K \implies \tau\cap\sigma \in K \ \ \& \ \ \tau\cap\sigma < \tau \ \ \& \ \ \tau\cap\sigma < \sigma$.
\vspace{.0ex}
\end{enumerate}

For any {\cubicaltxt} complex $K \!\subset\! \Cubic(n)$, we denote $K^{q}=\{\sigma \!\in\! K \,;\, \text{$\sigma$ is a $q$-{\cubictxt}} \}$, $n \!\geq\! 0$ and $\vert{K}\vert = \underset{\sigma \in K}{\cup}\sigma$.
For any {\cubicaltxt} complexes $K$ and $L$, a map $f : \vert{L}\vert \to \vert{K}\vert$ in $\convex$ is called `polyhedral' w.r.t. $L$ and $K$, if $f(\sigma) \in K$ for any $\sigma \in L$.
If a {\cubicaltxt} complex $K \!\subset\! \Cubic(n)$ satisfies $\vert{K}\vert=\simp^{n}$, we call $K$ a `{\cubicaltxt} subdivision' of an $n$-cube $\simp^{n}$.

\begin{defn}
We define a category $\subdivision_{\mathcal{U}}$ as follows:
\begin{description}
\item[Object]
$\Object{\subdivision_{\mathcal{U}}} = \vspace{.5ex}\{(K,P) \in \Cubic(n){\times}\mathcal{E}^{X}(\simp^{n}) \,;\, \vert{K}\vert\!=\!\simp^{n}, \ \forall_{\sigma \in K}\,P\vert_{\sigma} \!\in\! \mathcal{E}^{\,\mathcal{U}}, \ n \!\geq\!0\}$,
\item[Morphism]
$\subdivision_{\mathcal{U}}((L,Q),(K,P)) = \{f : \vert{L}\vert \subset \vert{K}\vert \ \text{polyhedral} \,;\, Q\!=\!P\vert_{\vert{L}\vert}\}$.
\end{description}
\end{defn}

Let $\subdivision_{X} = \subdivision_{\{\!X\!\}}$.
Then there is an embedding $\iota^{\,\mathcal{U}}_{SD} : \subdivision_{\mathcal{U}} \hookrightarrow \subdivision_{X}$.

\begin{thm}\label{thm:subdivision}
There is a functor $\subdiv^{*}_{\mathcal{U}} : \subdivision_{X} \twoheadrightarrow \subdivision_{\mathcal{U}}$ such that $\subdiv^{*}_{\mathcal{U}}{\circ}\iota^{\,\mathcal{U}}_{SD} = \id$.\vspace{-.5ex}
\end{thm}
\begin{proof}
We construct a functor $\subdiv_{\mathcal{U}} : \subdivision_{X} \to \subdivision_{X}$ satisfying $\subdiv_{\mathcal{U}}{\circ}\iota^{\,\mathcal{U}}_{SD}$ $=$ $\iota^{\,\mathcal{U}}_{SD}$.

Firstly, for $(K,P) \in \mathcal{O}(\subdivision_{X})$, $K \subset C(n)$ is a {\cubicaltxt} subdivision of $\vert{K}\vert=
\simp^{n} = \Dom{P}$.
Let $K_{P}(\mathcal{U})=\{\sigma \!\in\! K \,;\, \exists_{U \in \mathcal{U}}\,P(\sigma) \!\subset\! U\} < K$.
We define $\subdiv_{\mathcal{U}}(K,P)=(\subdiv^{\mathcal{U}}_{P}(K),P)$ by induction on dimension of a {\cubictxt} in $K$.
\par\vskip1.2ex\noindent
\hfil$\begin{array}{rl}
\subdiv^{\mathcal{U}}_{P}(K)^{0} \!\!&= K^{0} \cup \{\mathbold{b}_{\sigma} \,;\,\sigma \in K \!\smallsetminus\! K_{P}(\mathcal{U})\},
\\[1.2ex]
\subdiv^{\mathcal{U}}_{P}(K)^{q} \!\!&=
K_{P}(\mathcal{U})^{q} \ \cup \ \{\rho\ast{\mathbold{b}_{\sigma} \,;\, \rho \in \subdiv^{\mathcal{U}}_{P}(\partial{\sigma})^{q-1}}, \ \sigma \in K \!\smallsetminus\! K_{P}(\mathcal{U})\},
\end{array}$\hfil
\par\vskip1.2ex\noindent
where $\partial{\sigma}$ denotes the subcomplex $\{\tau\in K \,;\, \tau\!<\!\sigma\}$ of $K$.

Secondly, for any map $f : (L,Q) \to (K,P)$, we have $L\!\subset\!K$ and $Q\!=\!P\vert_{\vert{L}\vert}$.
Then by definition, we have $\subdiv_{\mathcal{U}}(L) \!\subset\!\subdiv_{\mathcal{U}}(K)$, and hence the inclusion $f : \vert{\subdiv_{\mathcal{U}}(L)}\vert = \vert{L}\vert \subset \vert{K}\vert = \vert{\subdiv_{\mathcal{U}}(K)}\vert$ is again polyhedral.
Thus we obtain $\subdiv_{\mathcal{U}}(f) = f : \subdiv_{\mathcal{U}}(L,Q) \to \subdiv_{\mathcal{U}}(K,P)$.

Thirdly, we give a distance of subcomplexes $K$ and $K_{P}(\mathcal{U})$ defined as follows:
\par\vskip1.5ex\noindent
\hfil$\begin{array}{l}
\varepsilon^{\mathcal{U}}_{P}(K) = \Min\left\{\diff{(\tau,\mathbold{x})} \,\midvert\, \tau \subset P^{-1}(U) \not\ni \mathbold{x}, \ U \in \mathcal{U} \ \ \& \ \text{$\tau$ is maximal in $K_{P}(\mathcal{U})$}\right\},
\\[1.5ex]
\diff^{\mathcal{U}}_{P}(K) = \Max\left\{\diff(\tau,\mathbold{x})\,\midvert\,\tau \cap \sigma \not= \emptyset, \ \mathbold{x} \in \sigma \in K\ \& \ \text{$\tau$ is maximal in $K_{P}(\mathcal{U})$}\right\},
\end{array}$\hfil
\par\vskip1.5ex\noindent
where $\diff{(\tau,\mathbold{x})}$ denotes the distance in $\simp^{n}$ of $\tau$ and $\mathbold{x}$, and hence $\varepsilon^{\mathcal{U}}_{P}(K) > 0$.
We can easily see that $\diff_{\mathcal{U}}(\subdiv^{\mathcal{U}}_{P}(K)) \leq \frac{n}{n{+}1}\diff_{\mathcal{U}}(K)$ and hence that, for sufficiently large $r \!>\! 0$, the $r$-times iteration of $\subdiv^{\mathcal{U}}_{P}$ satisfies $\diff^{\mathcal{U}}_{P}((\subdiv^{\mathcal{U}}_{P})^{r}(K)) < \varepsilon^{\mathcal{U}}_{P}(K)$.
Thus $\subdiv^{r}_{\mathcal{U}}(K,P) \in \subdivision_{\mathcal{U}}$.

Finally, when $(K,P) \in \subdivision_{\mathcal{U}}$, we have $\subdiv^{\mathcal{U}}_{P}(K,P) = (K,P)$ by definition, and hence $\subdiv^{*}_{\mathcal{U}}$ the sufficiently many times iteration of $\subdiv_{\mathcal{U}}$ on each $(K,P)$ is a desired functor.\vspace{-1ex}
\end{proof}

\begin{defn}
A functor $\tentdiv_{\mathcal{U}} : \subdivision_{X} \rightarrow \subdivision_{X}$ given by $\tentdiv_{\mathcal{U}}(K,P) = (\tentdiv^{\mathcal{U}}_{P}(K),\hat{P})$ for $(K,P) \in \Object{\subdivision_{X}}$ is defined as follows: we denote $\hat{P} = P{\circ}\proj_{1} : \simp^{n}{\times}I \to X$ which is a plot in $\mathcal{E}^{X}(\simp^{n+1})$. 
Then a {\cubicaltxt} subdivision $\tentdiv^{\mathcal{U}}_{P}(K)$ of \,${\simp}^{n+1}$ is defined as follows:
\par\vskip1.5ex\noindent
\hfil$\begin{array}{rl}
\tentdiv^{\mathcal{U}}_{P}(K)^{0} \!\!&= K^{0}{\times}\{0\} \cup \subdiv^{\mathcal{U}}_{P}(K)^{0}{\times}\{1\}, 
\\[1.2ex]
\tentdiv^{\mathcal{U}}_{P}(K)^{q} \!\!&=
K^{q}{\times}\{0\} \cup \subdiv^{\mathcal{U}}_{P}(K)^{q}{\times}\{1\} \cup K_{P}(\mathcal{U})^{q-1}{\times}I
\\[1ex]&\quad
\ \cup \ 
\{\rho\ast{(\mathbold{b}_{\sigma},1)} \,;\, \rho \in \tentdiv^{\mathcal{U}}_{P}(\partial{\sigma})^{q-1}, \ \sigma \in K \!\smallsetminus\! K_{P}(\mathcal{U})\}. 
\end{array}$\hfil
\par\vskip1.5ex\noindent
Also for a map $f : (L,Q) \to (K,P)$, we have $L\!\subset\!K$ and $Q\!=\!P\vert_{\vert{L}\vert}$.
Then by definition, we have $\tentdiv_{\mathcal{U}}(L) \!\subset\!\tentdiv_{\mathcal{U}}(K)$, and hence the inclusion $f{\times}\id : \vert{\tentdiv_{\mathcal{U}}(L)}\vert = \vert{L}\vert{\times}I \subset \vert{K}\vert{\times}I = \vert{\tentdiv_{\mathcal{U}}(K)}\vert$ is again polyhedral.
Thus we obtain $\tentdiv_{\mathcal{U}}(f) = f : \tentdiv_{\mathcal{U}}(L,Q) \to \tentdiv_{\mathcal{U}}(K,P)$.\vspace{-1ex}
\end{defn}

\begin{defn}
For any {\cubicaltxt} differential $p$-form $\omega \in \Omega_{\Cubical}^{p}(\mathcal{U})$, we have a {\cubicaltxt} differential $p$-form $\widetilde\omega \in \Omega_{\Cubical}^{p}(\mathcal{U})$ defined by $\widetilde\omega_{\underline{n}}(P) = (\lambda^{n})^{*}\omega_{\underline{n}}(P)$ for any $P \in \mathcal{E}_{\underline{n}}^{\,\mathcal{U}}$, $\lambda = \lambda_{0,1}$.
In addition, if $\omega$ is a differential $p$-form with compact support, then so is $\widetilde\omega$.\vspace{-.5ex}
\end{defn}
\begin{lem}\label{lem:de-rham-stability}
There is a homomorphism $D_{\mathcal{U}} : \Omega_{\Cubical}^{\ast}(\mathcal{U}) \to \Omega_{\Cubical}^{\ast}(\mathcal{U})$ such that $\diff{D_{\mathcal{U}}(\omega)_{\underline{n}}} + D_{\mathcal{U}}(\diff\omega)_{\underline{n}} = \widetilde\omega_{\underline{n}}-\omega_{\underline{n}}$ and $D_{\mathcal{U}}(\Omega_{\Cubical_{c}}^{p}(\mathcal{U})) \subset \Omega_{\Cubical_{c}}^{p-1}(\mathcal{U})$ for any $p \!\geq\! 0$.
\end{lem}
\begin{proof}
Let $H : I{\times}I \to I$ be a smooth homotopy between $\id : I \to I$ and $\lambda : I \to I$, which gives rise to a smooth homotopy $H_{n} : \simp^{n+1} = I {\times}\simp^{n} \to \simp^{n}$ of $\id : \simp^{n} \to \simp^{n}$ and $\lambda^{n} : \simp^{n} \to \simp^{n}$, $n \!\geq\! 0$.
Then we have $H_{n}{\circ}\incl_{0}=\id$ and $H_{n}{\circ}\incl_{1} = \lambda^{n}$, where $\incl_{t} : \simp^{n} \hookrightarrow I{\times}\simp^{n}$ is given by $\incl_{t}(\mathbold{x}) = (t,\mathbold{x})$.
For any {\cubicaltxt} differential $p$-form $\omega : \mathcal{E}^{\,\mathcal{U}}_{\Cubical} \to \wedge^{p}_{\Cubical}$, a {\cubicaltxt} ($p{-}1$)-form $D_{\mathcal{U}}(\omega) : \mathcal{E}^{\,\mathcal{U}}_{\Cubical} \to \wedge^{p-1}_{\Cubical}$ is defined on a plot $P \in \mathcal{E}_{\Cubical}^{\,\mathcal{U}}$, by the following formula.
\par\vskip1ex\noindent
\hfil$\displaystyle
\begin{array}{l}
D_{\mathcal{U}}(\omega)_{\underline{n}}(P) = \displaystyle\int_{I}H^{\ast}\omega_{\underline{n}}(P) : \simp^{n} \to \wedge^{p-1}(T^{*}_{n}),
\\[2ex]
\displaystyle \left[\int_{I}H^{\ast}\omega_{\underline{n}}(P)\right](\mathbold{x}) = \underset{i_{2},\cdots,i_{p}}{\textstyle\sum}\int_{0}^{1}a_{i_{2},\cdots,i_{p}}(t,\mathbold{x})\diff{t}{\cdot}\diff{x_{i_{2}}}\wedge\cdots\wedge\diff{x_{i_{p}}},
\end{array}
$\hfil
\par\vskip1.5ex\noindent
where we assume $H^{\ast}\omega_{\underline{n}}(P)$ $=$ $\underset{i_{2},\cdots,i_{p}}{\textstyle\sum}a_{i_{2},\cdots,i_{p}}(t,\mathbold{x})\diff{t}\wedge\diff{x_{i_{2}}}\wedge\cdots\wedge\diff{x_{i_{p}}} + \underset{i_{1},\cdots,i_{p}}{\textstyle\sum}b_{i_{1},\cdots,i_{p}}(t,\mathbold{x})$ $\diff{x_{i_{1}}}\wedge\cdots\wedge\diff{x_{i_{p}}} : I{\times}\simp^{n} \to \wedge^{p-1}(T^{*}_{n+1})$, $(t,\mathbold{x}) \in I{\times}\simp^{n}$ and $T^{*}_{n+1}=\real\diff{t}{\oplus}\overset{n}{\underset{i=1}{\oplus}}\real\diff{x_{i}}$.

First, let $\incl_{t} : \simp^{n} \to I{\times}\simp^{n}$ be the inclusion defined by $\incl_{t}(\mathbold{x})=(t,\mathbold{x})$ for $t=0, 1$.
By $H{\circ}\incl_{0} = \id$, we have $\omega_{\underline{n}}(P) = \id^{*}\omega_{\underline{n}}(P) = \incl_{0}^{*}H^{\ast}\omega_{\underline{n}}(P) = \underset{i_{1},\cdots,i_{p}}{\textstyle\sum}b_{i_{1},\cdots,i_{p}}(0,\mathbold{x}) \diff{x_{i_{1}}}\wedge\cdots\wedge\diff{x_{i_{p}}}$.
On the other hand by $H{\circ}\incl_{1} = \lambda^{n}$, we have $(\lambda^{n})^{*}\omega_{\underline{n}}(P) = \incl_{1}^{*}H^{\ast}\omega_{\underline{n}}(P) = \underset{i_{1},\cdots,i_{p}}{\textstyle\sum}b_{i_{1},\cdots,i_{p}}(1,\mathbold{x}) \diff{x_{i_{1}}}\wedge\cdots\wedge\diff{x_{i_{p}}}$\, for any $\mathbold{x} \in \simp^{n}$.

Second, by definition, we have
$\diff{H^{\ast}\omega_{\underline{n}}(P)}$ $=$ $\displaystyle\underset{i}{\textstyle\sum}\underset{i_{2},\cdots,i_{p}}{\textstyle\sum}\pder{a_{i_{2},\cdots,i_{p}}}by{x_{i}}(t,\mathbold{x})$ $\diff{x_{i}}\wedge\diff{t}\wedge\diff{x_{i_{2}}}\wedge\cdots\wedge\diff{x_{i_{p}}} + \!\!\!\underset{i_{1},\cdots,i_{p}}{\textstyle\sum}\displaystyle\pder{b_{i_{1},\cdots,i_{p}}}by{t}(t,\mathbold{x}) \displaystyle\diff{t}\wedge\diff{x_{i_{1}}}\wedge\cdots\wedge\diff{x_{i_{p}}} + \underset{i}{\textstyle\sum}\underset{i_{1},\cdots,i_{p}}{\textstyle\sum}\displaystyle\pder{b_{i_{1},\cdots,i_{p}}}by{x_{i}}(t,\mathbold{x}) \diff{x_{i}}\wedge\diff{x_{i_{1}}}\wedge\cdots\wedge\diff{x_{i_{p}}}$, 
and hence we obtain \vspace{.5ex}
$D_{\mathcal{U}}(\diff\omega)_{\underline{n}}(P) = \displaystyle\int_{I}H^{\ast}\diff\omega_{\underline{n}}(P) = -\displaystyle\underset{i}{\textstyle\sum}\underset{i_{2},\cdots,i_{p}}{\textstyle\sum}\int_{I}\pder{a_{i_{2},\cdots,i_{p}}}by{x_{i}}(t,\mathbold{x})\diff{t}{\cdot}\diff{x_{i}}\wedge\diff{x_{i_{2}}}\wedge\cdots\wedge\diff{x_{i_{p}}} + \underset{i_{1},\cdots,i_{p}}{\textstyle\sum}\displaystyle\int_{I}\pder{b_{i_{1},\cdots,i_{p}}}by{t}(t,\mathbold{x})\diff{t}{\cdot}\diff{x_{i_{1}}}\wedge\cdots\wedge\diff{x_{i_{p}}}$, $(t,\mathbold{x}) \in I{\times}\simp^{n}$.

Third, we have \vspace{.5ex}$D_{\mathcal{U}}(\omega)_{\underline{n}}(P) = \displaystyle\underset{i_{2},\cdots,i_{p}}{\textstyle\sum}\int_{I}a_{i_{2},\cdots,i_{p}}(t,\mathbold{x})\diff{t}{\cdot}\diff{x_{i_{2}}}\wedge\cdots\wedge\diff{x_{i_{p}}}$, and hence we obtain 
$\diff{D_{\mathcal{U}}(\omega)_{\underline{n}}}(P) = \displaystyle\underset{i}{\textstyle\sum}\underset{i_{2},\cdots,i_{p}}{\textstyle\sum}\int_{I}\pder{a_{i_{2},\cdots,i_{p}}}by{x_{i}}(t,\mathbold{x})\diff{t}{\cdot}\diff{x_{i}}\wedge\diff{x_{i_{2}}}\wedge\cdots\wedge\diff{x_{i_{p}}}$, $(t,\mathbold{x}) \in I{\times}\simp^{n}$.

Hence \vspace{.5ex}$\left[\diff{D_{\mathcal{U}}(\omega)_{\underline{n}}}(P) + D_{\mathcal{U}}(\diff\omega)_{\underline{n}}(P)\right](\mathbold{x}) = \underset{i_{1},\cdots,i_{p}}{\textstyle\sum}\displaystyle\int_{I}\pder{b_{i_{1},\cdots,i_{p}}}by{t}(t,\mathbold{x})\diff{t}{\cdot}\diff{x_{i_{1}}}\wedge\cdots\wedge\diff{x_{i_{p}}} = \underset{i_{1},\cdots,i_{p}}{\textstyle\sum} b_{i_{1},\cdots,i_{p}}(1,\mathbold{x}) \diff{x_{i_{1}}}\wedge\cdots\wedge\diff{x_{i_{p}}} - \underset{i_{1},\cdots,i_{p}}{\textstyle\sum} b_{i_{1},\cdots,i_{p}}(0,\mathbold{x}) \diff{x_{i_{1}}}\wedge\cdots\wedge\diff{x_{i_{p}}}$, $\mathbold{x} \in \simp^{n}$.\vspace{.5ex}
Thus we obtain $\diff{D_{\mathcal{U}}(\omega)}(P) + D_{\mathcal{U}}(\diff\omega)(P) = \widetilde\omega(P) - \omega(P)$.
By the above construction of $D_{\mathcal{U}}$, it is clear to see $D_{\mathcal{U}}(\Omega_{\Cubical_{c}}^{p}(\mathcal{U})) \subset \Omega_{\Cubical_{c}}^{p-1}(\mathcal{U})$, and it completes the proof of the lemma.\vspace{-1.0ex}
\end{proof}

\begin{rem}
We have $b_{i_{1},\cdots,i_{p}}(1,\mathbold{x})$ $=$ $b_{i_{1},\cdots,i_{p}}(0,\lambda^{n}(\mathbold{x})) \lambda'(x_{i_{1}}){\cdot}\cdots{\cdot}\lambda'(x_{i_{p}})$ for $1 \!\leq\! i_{1} \!<\! \ldots \!<\! i_{p} \leq n$ and $\mathbold{x}=(x_{1},\ldots,x_{n}) \in \simp^{n}$,
since $(\lambda^{n})^{*}\omega_{\underline{n}}(P) = \incl_{1}^{*}H^{\ast}\omega_{\underline{n}}(P)$.\vspace{-1.5ex}
\end{rem}
Let $\omega \in \Omega^{\ast}_{\Cubical}(X)$ and $P \in \mathcal{E}^{X}(\simp^{n})$. 
Then a {\cubicaltxt} complex $K=\{\sigma \,;\, \sigma \!<\! \simp^{n}\}$ derives {\cubicaltxt} subdivisions $K_{r}=(\subdiv^{\mathcal{U}}_{P})^{r}(K)$ and $K_{*}=(\subdiv^{\mathcal{U}}_{P})^{\ast}(K)$ where $K_{*}=K_{r}$ for sufficiently large $r \!\geq\! 0$. 
We define $\omega^{(r)} \in \Omega_{\Cubical}^{p}(\mathcal{U})$, $r \!\geq\! 0$, as follows: for any $\sigma \in K_{r}$, 
\par\vskip1.0ex\noindent
\hfil$
\omega^{(r)}_{\underline{n}}(P)\vert_{\Int\sigma} 
= \hat\omega^{(r)}_{\sigma}(P\vert_{\sigma})\vert_{\Int\sigma},
$\hfil
\par\vskip.5ex\noindent
where $\hat\omega^{(r)}_{\sigma}(P\vert_{\sigma})\vert_{\Int\sigma} = \omega_{\underline{n}}(P\vert_{\sigma}{\circ}\phi_{\sigma}){\circ}\lambda^{n}{\circ}\phi_{\sigma}^{-1} : \Int{\sigma} \overset{\phi_{\sigma}^{-1}}{\homeo} \Int{{\simp}^{n}} \overset{\lambda^{n}}{\rightarrow} \Int{{\simp}^{n}} \xrightarrow{\omega_{\underline{n}}(P{\circ}\phi_{\sigma})} \wedge^{p}$.
Then by definition, $\omega^{(r)}_{\underline{n}}(P)\vert_{\Int\sigma}$ can be smoothly extended to $\partial\sigma$, and hence $\omega^{(r)}_{\underline{n}}(P) : \simp^{n} \to \wedge^{p}_{T_{n}^{*}}$ is well-defined and we obtain $\omega^{(r)} \in \Omega_{\Cubical}^{p}(X)$. \vspace{-1.5ex}
\begin{lem}\label{lem:de-rham-injectivity}
There is a homomorphism \vspace{.2ex}$D^{(r)}_{\mathcal{U}} : \Omega_{\Cubical}^{\ast}(X)\to \Omega_{\Cubical}^{\ast}(X)$ such that $\diff{D^{(r)}_{\mathcal{U}}(\omega)} + D^{(r)}_{\mathcal{U}}(\diff\omega) 
= \omega^{(r+1)}-\omega^{(r)}$ and $D^{(r)}_{\mathcal{U}}(\Omega_{\Cubical_{c}}^{p}(\mathcal{U})) \subset \Omega_{\Cubical_{c}}^{p-1}(\mathcal{U})$ for $p \!\geq\! 0$.\vspace{-1ex}
\end{lem}
\begin{proof}
For any $\omega \in \Omega_{\Cubical}^{p}(\mathcal{U})$, we define $D^{(r)}_{\mathcal{U}}(\omega) \in \Omega_{\Cubical}^{p}(X)$ as follows:
let $P \in \mathcal{E}^{X}(\simp^{n})$. 
We have a {\cubicaltxt} complex $K=\{\sigma \,;\, \sigma \!<\! \simp^{n}\}$ which derives {\cubicaltxt} subdivisions $K_{r}=(\subdiv^{\mathcal{U}}_{P})^{r}(K)$ of $\simp^{n}$ and $\widehat{K}_{r}=\tentdiv^{\mathcal{U}}_{P}(K_{r})$ of $I{\times}\simp^{n}$ so that $\incl_{0}^{*}\widehat{K}_{r}=K_{r}$ and $\incl_{1}^{*}\widehat{K}_{r}=K_{r+1}$.
Now we define a smooth function $\widehat\omega(P) : I{\times}\simp^{n} \to \wedge^{p}(T^{\ast}_{n+1})$ as follows: for any $\sigma \in \widehat{K}_{r}^{n+1}$, 
\par\vskip1.0ex\noindent
\hfil$
\widehat{\omega}(P)\vert_{\Int\sigma} 
= \widehat{\omega}'_{\sigma}(P{\circ}\proj_{2}\vert_{\sigma})\vert_{\Int\sigma} : I{\times}\simp^{n} \longrightarrow \wedge^{p}(T^{\ast}_{n+1}),
$\hfil
\par\vskip.5ex\noindent
where $\widehat{\omega}'_{\sigma}(P{\circ}\proj_{2}\vert_{\sigma})\vert_{\Int\sigma} = \omega_{\underline{n{+}1}}(P{\circ}\proj_{2}\vert_{\sigma}{\circ}\phi_{\sigma}){\circ}\lambda^{n+1}{\circ}\phi_{\sigma}^{-1} : \Int{\sigma} \overset{\phi_{\sigma}^{-1}} \homeo \Int{{\simp}^{n+1}} \overset{\lambda^{n+1}}{\homeo} \Int{{\simp}^{n+1}}$ $\xrightarrow{\omega_{\underline{n+1}}(P{\circ}{\smallproj_{2}}{\circ}\phi_{\sigma})}$ $\wedge^{p}$.
Then by definition, $\widehat{\omega}'_{\sigma}(P{\circ}\proj_{2}\vert_{\sigma})\vert_{\Int\sigma} $ can be smoothly extended to $\sigma$ and we obtain a smooth function $\widehat{\omega}(P) : I{\times}\simp^{n} \to \wedge^{p}_{T_{n+1}^{*}}$. 

First, a {\cubicaltxt} ($p{-}1$)-form $D^{(r)}_{\mathcal{U}}(\omega)  \in \Omega_{\Cubical}^{p-1}(X)$ is defined as follows: for any {\cubicaltxt} differential $p$-form $\omega : \mathcal{E}^{X}_{\Cubical} \to \wedge^{p}_{\Cubical}$ on a plot $P \in \mathcal{E}_{\Cubical}^{X}$, 
\par\vskip1.0ex\noindent
\hfil$\displaystyle
\begin{array}{l}
D^{(r)}_{\mathcal{U}}(\omega)_{\underline{n}}(P) = \displaystyle\int_{I}\widehat{\omega}(P) : \simp^{n} \to \wedge^{p-1}(T^{*}_{n}),
\\[2ex]
\displaystyle \left[\int_{I}\widehat{\omega}(P)\right](\mathbold{x}) = \underset{i_{2},\cdots,i_{p}}{\textstyle\sum}\int_{0}^{1}a_{i_{2},\cdots,i_{p}}(t,\mathbold{x})\diff{t}{\cdot}\diff{x_{i_{2}}}\wedge\cdots\wedge\diff{x_{i_{p}}},
\end{array}
$\hfil
\par\vskip1.5ex\noindent
where $\widehat{\omega}(P)=\underset{i_{2},\cdots,i_{p}}{\textstyle\sum}a_{i_{2},\cdots,i_{p}}(t,\mathbold{x})\diff{t}\wedge\diff{x_{i_{2}}}\wedge\cdots\wedge\diff{x_{i_{p}}} + \underset{i_{1},\cdots,i_{p}}{\textstyle\sum}b_{i_{1},\cdots,i_{p}}(t,\mathbold{x})$ $\diff{x_{i_{1}}}\wedge\cdots\wedge\diff{x_{i_{p}}} : I{\times}\simp^{n} \to \wedge^{p-1}(T^{*}_{n+1})$, $(t,\mathbold{x}) \in I{\times}\simp^{n}$ and $T^{*}_{n+1}=\real\diff{t}{\oplus}\overset{n}{\underset{i=1}{\oplus}}\real\diff{x_{i}}$.
Then, since $\incl_{0}^{*}\widehat{K}_{r}=K_{r}$ and $\incl_{1}^{*}\widehat{K}_{r}=K_{r+1}$, we easily see that $\omega^{(r)}_{\underline{n}}(P) = \incl_{0}^{*}\,\widehat\omega(P) = \underset{i_{1},\cdots,i_{p}}{\textstyle\sum}b_{i_{1},\cdots,i_{p}}(0,\mathbold{x}) \diff{x_{i_{1}}}\wedge\cdots\wedge\diff{x_{i_{p}}}$ and $\omega^{(r+1)}_{\underline{n}}(P) = \incl_{1}^{*}\,\widehat\omega(P) = \underset{i_{1},\cdots,i_{p}}{\textstyle\sum}b_{i_{1},\cdots,i_{p}}(1,\mathbold{x}) \diff{x_{i_{1}}}\wedge\cdots\wedge\diff{x_{i_{p}}}$.

Second, by definition, we have
$\widehat{\diff{\omega}}(P) = \diff{\widehat\omega}(P) = \displaystyle\underset{i}{\textstyle\sum}\underset{i_{2},\cdots,i_{p}}{\textstyle\sum}\pder{a_{i_{2},\cdots,i_{p}}}by{x_{i}}(t,\mathbold{x})$ $\diff{x_{i}}\wedge\diff{t}\wedge\diff{x_{i_{2}}}\wedge\cdots\wedge\diff{x_{i_{p}}} + \!\!\!\underset{i_{1},\cdots,i_{p}}{\textstyle\sum}\displaystyle\pder{b_{i_{1},\cdots,i_{p}}}by{t}(t,\mathbold{x}) \displaystyle\diff{t}\wedge\diff{x_{i_{1}}}\wedge\cdots\wedge\diff{x_{i_{p}}} + \underset{i}{\textstyle\sum}\underset{i_{1},\cdots,i_{p}}{\textstyle\sum}\displaystyle\pder{b_{i_{1},\cdots,i_{p}}}by{x_{i}}(t,\mathbold{x}) \diff{x_{i}}\wedge\diff{x_{i_{1}}}\wedge\cdots\wedge\diff{x_{i_{p}}}$, 
and hence  
$D^{(r)}_{\mathcal{U}}(\diff\omega)_{\underline{n}}(P) = \displaystyle\int_{I}\widehat{\diff{\omega}}(P) = -\displaystyle\underset{i}{\textstyle\sum}\underset{i_{2},\cdots,i_{p}}{\textstyle\sum}\int_{I}\pder{a_{i_{2},\cdots,i_{p}}}by{x_{i}}(t,\mathbold{x})\diff{t}\cdot\diff{x_{i}}\wedge\diff{x_{i_{2}}}\wedge\cdots\wedge\diff{x_{i_{p}}} + \underset{i_{1},\cdots,i_{p}}{\textstyle\sum}\displaystyle\int_{I}\pder{b_{i_{1},\cdots,i_{p}}}by{t}(t,\mathbold{x})\diff{t}{\cdot}\diff{x_{i_{1}}}\wedge\cdots\wedge\diff{x_{i_{p}}}$, $(t,\mathbold{x}) \in I{\times}\simp^{n}$.

Third, we have $D^{(r)}_{\mathcal{U}}(\omega)_{\underline{n}}(P) = \displaystyle\underset{i_{2},\cdots,i_{p}}{\textstyle\sum}\int_{I}a_{i_{2},\cdots,i_{p}}(t,\mathbold{x})\diff{t}{\cdot}\diff{x_{i_{2}}}\wedge\cdots\wedge\diff{x_{i_{p}}}$,\vspace{.5ex} and hence we obtain 
$\diff{D^{(r)}_{\mathcal{U}}(\omega)_{\underline{n}}}(P) = \displaystyle\underset{i}{\textstyle\sum}\underset{i_{2},\cdots,i_{p}}{\textstyle\sum}\int_{I}\pder{a_{i_{2},\cdots,i_{p}}}by{x_{i}}(t,\mathbold{x})\diff{t}{\cdot}\diff{x_{i}}\wedge\diff{x_{i_{2}}}\wedge\cdots\wedge\diff{x_{i_{p}}}$, $(t,\mathbold{x}) \in I{\times}\simp^{n}$.

Hence \vspace{.5ex}$\left[\diff{D^{(r)}_{\mathcal{U}}(\omega)_{\underline{n}}}(P) + D^{(r)}_{\mathcal{U}}(\diff\omega)_{\underline{n}}(P)\right](\mathbold{x}) = \underset{i_{1},\cdots,i_{p}}{\textstyle\sum}\displaystyle\int_{I}\pder{b_{i_{1},\cdots,i_{p}}}by{t}(t,\mathbold{x})\diff{t}{\cdot}\diff{x_{i_{1}}}\wedge\cdots\wedge\diff{x_{i_{p}}}$ $=$ $\underset{i_{1},\cdots,i_{p}}{\textstyle\sum} b_{i_{1},\cdots,i_{p}}(1,\mathbold{x}) \diff{x_{i_{1}}}\wedge\cdots\wedge\diff{x_{i_{p}}} - \underset{i_{1},\cdots,i_{p}}{\textstyle\sum} b_{i_{1},\cdots,i_{p}}(0,\mathbold{x}) \diff{x_{i_{1}}}\wedge\cdots\wedge\diff{x_{i_{p}}}$, $\mathbold{x} \in \simp^{n}$.\vspace{.5ex}
Thus we obtain $\diff{D^{(r)}_{\mathcal{U}}(\omega)}(P) + D^{(r)}_{\mathcal{U}}(\diff\omega)(P) = \omega^{(r+1)}(P) - \omega^{(r)}(P)$.\vspace{.5ex}
By the above construction of $D^{(r)}_{\mathcal{U}}$, it is clear to see that $D^{(r)}_{\mathcal{U}}(\Omega_{\Cubical_{c}}^{p}(\mathcal{U})) \subset \Omega_{\Cubical_{c}}^{p-1}(\mathcal{U})$.\vspace{-1ex}
\end{proof}

\begin{thm}\label{thm:excision}
The restriction $\res : \Omega^{\ast}_{\Cubical}(X) \rightarrow \Omega^{\ast}_{\Cubical}(\mathcal{U})$ induces an isomorphism of {\cubicaltxt} de Rham cohomologies $\res^{*} : H^{\ast}_{\Cubical}(X) \to H^{\ast}_{\Cubical}(\mathcal{U})$.
In addition, $\res$ induces a map $\res : \Omega^{\ast}_{\Cubical_{c}}(X) \rightarrow \Omega^{\ast}_{\Cubical_{c}}(\mathcal{U})$ which further induces an isomorphism $\res^{*} : H^{\ast}_{\Cubical_{c}}(X) \to H^{\ast}_{\Cubical_{c}}(\mathcal{U})$.
\end{thm}
\begin{proof}
For any $\omega \in \Omega_{\Cubical}^{p}(\mathcal{U})$, we define $\omega^{\ast} \in \Omega_{\Cubical}^{p}(X)$ as follows:
let $P \in \mathcal{E}^{X}(\simp^{n})$. 
Then we obtain a {\cubicaltxt} complex $K=\{\sigma \,;\, \sigma \!<\! \simp^{n}\}$ which derives a {\cubicaltxt} subdivision $K_{*}=(\subdiv^{\mathcal{U}}_{P})^{\ast}(K)$. 
We define {\cubicaltxt} differential $p$-forms $\omega^{\ast} \in \Omega_{\Cubical}^{p}(\mathcal{U})$ as follows: for any $\sigma \in K_{*}$, 
\par\vskip1ex\noindent
\hfil$
\omega^{\ast}_{\underline{n}}(P)\vert_{\Int\sigma} 
= \hat\omega^{\ast}_{\sigma}(P\vert_{\sigma})\vert_{\Int\sigma},
$\hfil
\par\vskip.5ex\noindent
where $\hat\omega^{\ast}_{\sigma}(P\vert_{\sigma})\vert_{\Int\sigma} = \omega_{\underline{n}}(P\vert_{\sigma}{\circ}\phi_{\sigma}){\circ}\lambda^{n}{\circ}\phi_{\sigma}^{-1} : \Int{\sigma} \overset{\phi_{\sigma}^{-1}} \homeo \Int{{\simp}^{n}} \overset{\lambda^{n}}{\homeo} \Int{{\simp}^{n}} \xrightarrow{\omega_{\underline{n}}(P{\circ}\phi_{\sigma})} \wedge^{p}$.
Then by definition, $\omega^{\ast}_{\underline{n}}(P)\vert_{\Int\sigma}$ can be uniquely extended to $\partial\sigma$ and we obtain $\omega^{\ast}_{\underline{n}}(P) : \simp^{n} \to \wedge^{p}_{T_{n}^{*}}$ so that $\omega^{\ast} \in \Omega_{\Cubical}^{p}(X)$ whose restriction to $\Omega_{\Cubical}^{p}(\mathcal{U})$ equals, by definition, to $\widetilde\omega$ with a $(p{-}1)$-form $D_{\mathcal{U}}(\omega) \in \Omega_{\Cubical}^{p-1}(\mathcal{U})$ satisfying $\diff{D_{\mathcal{U}}(\omega)}  = \widetilde\omega-\omega$ if $\diff\omega=0$, by Lemma \ref{lem:de-rham-stability}.
If $\diff\omega=0$, then $\diff{\hat\omega^{\ast}}=0$, and hence $\diff{\omega^{\ast}}=0$.
Thus the restriction $\res : \Omega^{\ast}_{\Cubical}(X) \rightarrow \Omega^{\ast}_{\Cubical}(\mathcal{U})$ induces an epimorphism $\res^{*} : H^{\ast}_{\Cubical}(X) \to H^{\ast}_{\Cubical}(\mathcal{U})$ of {\cubicaltxt} de Rham cohomologies.

So we are left to show that $\res^{*} : H^{\ast}_{\Cubical}(X) \to H^{\ast}_{\Cubical}(\mathcal{U})$ is a monomorphism: let $\omega \in \Omega_{\Cubical}^{p}(X)$. 
Then we obtain a {\cubicaltxt} differential $p$-forms $\omega^{(r)} \in \Omega_{\Cubical}^{p}(\mathcal{U})$ and $\omega^{*} \in \Omega_{\Cubical}^{p}(\mathcal{U})$ so that $\omega^{(r)}=\omega^{*}$ for sufficiently large $r \geq 0$. 
By Lemma \ref{lem:de-rham-injectivity}, there is a $(p{-}1)$-form $D_{\mathcal{U}}^{(r)}(\omega) \in \Omega_{\Cubical}^{p-1}(X)$ such that $\diff{D_{\mathcal{U}}^{(r)}(\omega)} = \omega^{(r+1)}-\omega^{(r)}$ if $\diff\omega=0$.
If we assume $\res^{\ast}([\omega])=0$, then we may assume $\res(\omega)=0$ and $\diff\omega=0$,\vspace{.5ex} and so we obtain $\omega^{\ast}=0$ and $\omega=\diff\left\{\underset{r=0}{\overset{N}{\textstyle\sum}}\,D_{\mathcal{U}}^{(r)}(\omega)-D_{\!\{X\}\!}(\omega)\right\}$ for sufficiently large $N \!\geq\! 0$,\vspace{.5ex} in other words, $\omega$ is an exact form and cohomologous to zero.
Thus $\res^{*} : H^{\ast}_{\Cubical}(X) \to H^{\ast}_{\Cubical}(\mathcal{U})$ is an monomorphism. 
\end{proof}

\section{Mayer-Vietoris sequence and Theorem of de Rham}

\begin{thm}\label{thm:cubical-open-covering}
Let \,$\mathcal{U}=\{U_{1},U_{2}\}$ be any open covering of a differentiable space $X$. 
The canonical inclusions $i_{t} : U_{1} \cap U_{2} \hookrightarrow U_{t}$ and $j_{t} : U_{t} \hookrightarrow X$, $t=1, 2$, induce $\psi^{\natural} : \Omega_{\Cubical}^{p}(\mathcal{U}) \to \Omega_{\Cubical}^{p}(U_{1}) \oplus \Omega_{\Cubical}^{p}(U_{2})$ and $\phi^{\natural} : \Omega_{\Cubical}^{p}(U_{1}) \oplus \Omega_{\Cubical}^{p}(U_{2}) \to \Omega_{\Cubical}^{p}(U_{1}\cap U_{2})$ by $\psi^{\natural}(\omega) = i_{1}^{\sharp}\omega{\oplus}i_{2}^{\sharp}\omega$ and $\phi^{\natural}(\eta_{1}{\oplus}\eta_{2}) = j_{1}^{\sharp}\eta_{1}-j_{2}^{\sharp}\eta_{2}$. 
Then we obtain the following long exact sequence. 
\par\vskip1ex\noindent\hfil
$\begin{array}{l}
H_{\Cubical}^{0}(X) \rightarrow \cdots \rightarrow H_{\Cubical}^{p}(X) \xrightarrow{\psi^{*}} H_{\Cubical}^{p}(U_{1}){\oplus}H_{\Cubical}^{p}(U_{2}) \xrightarrow{\phi^{*}} H_{\Cubical}^{p}(U_{1} \cap U_{2})
\\[1.5ex]\qquad\qquad
\rightarrow H_{\Cubical}^{p+1}(X) \xrightarrow{\psi^{*}} H_{\Cubical}^{p+1}(U_{1}){\oplus}H_{\Cubical}^{p+1}(U_{2}) \xrightarrow{\phi^{*}} H_{\Cubical}^{p+1}(U_{1} \cap U_{2}) \rightarrow \cdots,
\end{array}$\hfil
\par\vskip1.5ex\noindent
where $\psi^{*}$ and $\phi^{*}$ are induced from $\psi^{\natural}$ and $\phi^{\natural}$.\vspace{-.5ex}
\end{thm}
\begin{proof}
Since $H_{\Cubical}^{\ast}(X) = H_{\Cubical}^{\ast}(\mathcal{U})$ by Theorem \ref{thm:excision}, we are left to show long exact sequence
\par\vskip0ex\noindent
\hfil$
0 \longrightarrow \Omega_{\Cubical}^{p}(\mathcal{U}) 
\overset{\psi^{\natural}}\longrightarrow \Omega_{\Cubical}^{p}(U_{1}) \oplus \Omega_{\Cubical}^{p}(U_{2}) 
\overset{\phi^{\natural}}\longrightarrow \Omega_{\Cubical}^{p}(U_{0}) \longrightarrow 0, \quad U_{0}=U_{1} \cap U_{2}.
$\hfil
\begin{description}
\item[(exactness at $\Omega_{\Cubical}^{p}(\mathcal{U})$)]
Assume $\psi^{\natural}(\omega)=0$, and so $j_{t}^{\sharp}\omega=0$ for $t=1, 2$.
Then for any $P : \simp^{n} \to X$, $P \in \mathcal{E}_{\Cubical}^{\,\mathcal{U}}$, we have either $\Im{P} \subset U_{1}$ or $\Im{P} \subset U_{2}$.
Therefore, we may assume either $P \in \mathcal{E}_{\Cubical}^{U_{0}}$ or $P \in \mathcal{E}_{\Cubical}^{U_{1}}$.
In each case, we have $\omega_{\underline{n}}(P) = 0$, which implies that $\omega=0$.
Thus $\psi^{\natural}$ is monic.
\vspace{.5ex}
\item[(exactness at $\Omega_{\Cubical}^{p}(U_{1}) \oplus \Omega_{\Cubical}^{p}(U_{2})$)]
Assume $\phi^{\natural}(\eta^{(1)}{\oplus}\eta^{(2)})=0$, and so $i_{1}^{\sharp}\eta^{(1)}=i_{2}^{\sharp}\eta^{(2)}$.
Then we can construct a differential $p$-form $\omega \in \Omega_{\Cubical}^{p}(\mathcal{U})$ as follows: for any $P \in \mathcal{E}_{\Cubical}^{\,\mathcal{U}}$, we have $\Im{P} \subset U_{t}$ for either $t=1$ or $2$.
Using this $t$, we define $\omega_{\underline{n}}(P)=\eta^{(t)}_{\underline{n}}(P)$.
If $\Im{P} \subset U_{1}$ and $\Im{P} \subset U_{2}$, then we have $\Im{P} \subset U_{1} \cap U_{2}$, and hence $\eta^{(1)}_{\underline{n}}(P)=\eta^{(2)}_{\underline{n}}(P)$, since $i_{1}^{\sharp}\eta^{(1)}=i_{2}^{\sharp}\eta^{(2)}$.
It implies that $\omega$ is well-defined and that $\psi^{\natural}(\omega) = \eta^{(1)}{\oplus}\eta^{(2)}$.
The converse is clear and we have $\Ker{\phi^{\natural}}=\Im{\psi^{\natural}}$.
\vspace{.5ex}
\item[(exactness at $\Omega_{\Cubical}^{p}(U_{0})$)]
Assume $\kappa \in \Omega_{\Cubical}^{p}(U_{0})$. 
We define $\kappa^{(t)} \in \Omega_{\Cubical}^{p}(U_{t})$, $t = 1, 2$ as follows:
for any $P_{t} \in \mathcal{E}_{\Cubical}^{U_{t}}$, we define $\kappa^{(t)}_{\underline{n}}(P_{t})(\mathbold{x})$ by $(-1)^{t-1}\rho^{(3{-}t)}_{P_{t}}(\mathbold{x}){\cdot}\kappa_{\underline{n}}(P_{t})(\mathbold{x})$ if $\mathbold{x} \in P_{t}^{-1}(U_{3-t})$ and by $0$ if $\mathbold{x} \not\in \Supp{\rho^{3-t}_{P_{t}}}$. Hence $\kappa^{(t)}$ is well-defined satisfying $i_{1}^{\sharp}\kappa^{(1)} - i_{2}^{\sharp}\kappa^{(2)} = \kappa$, and we obtain $\kappa = \phi^{\natural}(\kappa^{(1)}{\oplus}\kappa^{(2)})$.
Thus $\phi^{\natural}$ is an epimorphism.
\end{description}
Since $\psi^{\natural}$ and $\phi^{\natural}$ are clearly cochain maps, we obtain the desired long exact sequence.
\end{proof}

Now let us turn our attention to the differential forms with compact support.
Let $X=(X,\mathcal{E}^{X})$ be a weakly-separated differentiable space.
\begin{defn}
Let $U$ be an open set in $X$, $F \!\subset\! U$ a closed set in $X$ and \,$\mathcal{U}$ an open covering of $U$.
We denote by $\Omega^{p}_{\Cubical_{c}}(\mathcal{U};F)$ the set of all $\omega \in \Omega^{p}_{\Cubical_{c}}(\mathcal{U})$ satisfying $\Supp\omega_{\underline{n}}(P) \subset P^{-1}(F)$ for any $P \in \mathcal{E}(\simp^{n})$.
For example, any $\omega \in \Omega^{p}_{\Cubical_{c}}(\mathcal{U})$ is in $\Omega^{p}_{\Cubical_{c}}(\mathcal{U};F)$ if $F \!\supset\! K_{\omega}$.
We denote by $H^{\ast}_{\Cubical_{c}}(\mathcal{U};F)$ the cohomology of $\Omega^{\ast}_{\Cubical_{c}}(\mathcal{U};F)$ a differential subalgebra of $\Omega^{\ast}_{\Cubical_{c}}(\mathcal{U})$.
\end{defn}

\begin{defn}\label{defn:cubical-induced-map-1}
Let $U$ and $V$ be open sets and $F \!\subset\! U$ and $G \!\subset\! V$ be closed sets in $X$ so that $(U,F) \subset (V,G)$, and $j : (U,F) \hookrightarrow (V,G)$ be the canonical inclusion.
Let \,$\mathcal{U}$ and $\mathcal{V}$ be open coverings of $U$ and $V$, respectively, satisfying $F \cap W = \emptyset$ for any $W \in \mathcal{V}\smallsetminus\mathcal{U}$.
Then a homomorphism $j_{\sharp} : \Omega^{p}_{\Cubical_{c}}(\mathcal{U};F) \to \Omega^{p}_{\Cubical_{c}}(\mathcal{V};G)$ is defined as follows: for any $\omega \in \Omega^{p}_{\Cubical_{c}}(\mathcal{U};F)$, $j_{\sharp}\omega \in \Omega^{p}_{\Cubical_{c}}(\mathcal{V};G)$ is given, for $Q \in \mathcal{E}^{\mathcal{V}}(\simp^{m})$, by
\par\vskip.5ex\noindent
\hfil$
\begin{cases}
(j_{\sharp}\omega)_{\underline{m}}(Q) = \omega_{\underline{m}}(Q), & \text{if \ $\Im{Q} \subset W$ for some $W \!\in \mathcal{U}$,}
\\[1ex]
(j_{\sharp}\omega)_{\underline{m}}(Q) = 0,
& \text{if \ $\Im{Q} \subset W$ for some $W \!\in \mathcal{V} \!\smallsetminus\! \mathcal{U}$}
\end{cases}
$\hfil
\par\vskip1.0ex\noindent
with $K_{j_{\sharp}\omega} = K_{\omega} \subset F \subset G$.
In particular, for any $\omega \in \Omega^{p}_{\Cubical_{c}}(\mathcal{U})$, we have $\omega \in \Omega^{p}_{\Cubical_{c}}(\mathcal{U};K_{\omega})$, and so we obtain $j_{\sharp}\omega \in \Omega^{p}_{\Cubical_{c}}(j_{\sharp}\,\mathcal{U}_{\omega};K_{\omega}) \subset \Omega^{p}_{\Cubical_{c}}(j_{\sharp}\,\mathcal{U}_{\omega})$, \ $j_{\sharp}\,\mathcal{U}_{\omega}=\mathcal{U}\cup \{V \!\smallsetminus\! K_{\omega}\}$.\vspace{-.5ex}
\end{defn}
\begin{rem}
In Definition \ref{defn:cubical-induced-map-1}, the map $j_{\sharp}$ induced from $j : (U,F) \hookrightarrow (V,G)$ satisfies that $(j_{\sharp}\omega)_{\underline{m}}(j{\circ}Q) = \omega_{\underline{m}}(Q)$ for any $m \!\geq\! 0$ and $Q \in \mathcal{E}^{\mathcal{U}}(\simp^{m})$.\vspace{-.5ex}
\end{rem}

\begin{prop}\label{prop:zero-extension}
Let $X=(X,\mathcal{E}^{X})$ be a weakly-separated differentiable space and $U$ and $V$ open in $X$.
Then the correspondence $\Omega_{{\Cubical}_{c}}^{\ast}(U) \ni \omega \mapsto j_{\sharp}\omega \in \Omega_{\Cubical_{c}}^{\ast}(j_{\sharp}\,\mathcal{U}_{\omega})$ induced from the canonical inclusion $j : U \hookrightarrow V$ induces a homomorphism $j_{*} : H_{{\Cubical}_{c}}^{*}(U) \to H_{{\Cubical}_{c}}^{*}(V)$, since there is a canonical isomorphism $H_{{\Cubical}_{c}}^{*}(j_{\sharp}\,\mathcal{U}_{\omega}) \cong H_{{\Cubical}_{c}}^{*}(V)$ by Theorem \ref{thm:excision}.
\end{prop}
\begin{proof}
Let $\omega, \eta \in \Omega_{{\Cubical}_{c}}^{\ast}(U)$. 
Then $K=K_{\omega} \cup K_{\eta}$ is compact in $U$ and hence in $X$.
Let $\mathcal{U}=\{U,V \smallsetminus K\}$, which is a finer open covering of $\mathcal{U}_{\omega}$ and $\mathcal{U}_{\eta}$, and hence both isomorphisms $H_{{\Cubical}_{c}}^{*}(V) \to H_{{\Cubical}_{c}}^{*}(\mathcal{U}_{\omega})$ and $H_{{\Cubical}_{c}}^{*}(V) \to H_{{\Cubical}_{c}}^{*}(\mathcal{U}_{\eta})$ defined in Theorem \ref{thm:excision} go through the isomorphism $H_{{\Cubical}_{c}}^{*}(V) \to H_{{\Cubical}_{c}}^{*}(\mathcal{U})$.
Thus the homomorphisms $H_{{\Cubical}_{c}}^{*}(\mathcal{U}) \to H_{{\Cubical}_{c}}^{*}(\mathcal{U}_{\omega})$ and $H_{{\Cubical}_{c}}^{*}(\mathcal{U}) \to H_{{\Cubical}_{c}}^{*}(\mathcal{U}_{\eta})$ are also isomorphisms.
By definition, $j_{\sharp}(\omega{+}\eta) = j_{\sharp}(\omega)+j_{\sharp}(\eta)$ in $\Omega_{{\Cubical}_{c}}^{\ast}(\mathcal{U})$, and hence $j_{*}([\omega{+}\eta]) = j_{*}([\omega])+j_{*}([\eta])$ in $H_{{\Cubical}_{c}}^{*}(X)$ for any $[\omega], [\eta] \in H_{{\Cubical}_{c}}^{*}(U)$.\vspace{-.5ex}
\end{proof}

\begin{thm}\label{thm:cubical-open-covering2}
Let \,$\mathcal{U}=\{U_{1},U_{2}\}$ be an open covering of a weakly-separated differentiable space $X$ with a normal partition of unity $\{\rho^{(1)},\rho^{(2)}\}$ belonging to \,$\mathcal{U}$, i.e., there are closed subsets $\{G_{1},G_{2}\}$ such that $G_{t} \subset U_{t}$ and $\Supp\rho_{\underline{n}}^{(t)}(P) \subset P^{-1}(G_{t})$ for any $P \in \mathcal{E}(\simp^{n})$, $t=1, 2$.
Then we have $G_{1} \cup G_{2} = X$.
Let $G_{0}=G_{1} \cap G_{2} \subset U_{0}=U_{1} \cap U_{2}$.
The canonical inclusions $i_{t} : U_{1} \cap U_{2} \hookrightarrow U_{t}$ and $j_{t} : U_{t} \hookrightarrow X$, $t=1, 2$, induce $\phi_{*} : H_{{\Cubical}_{c}}^{p}(U_{0}) \to H_{{\Cubical}_{c}}^{p}(U_{1}) \oplus H_{{\Cubical}_{c}}^{p}(U_{2})$ and $\psi_{*} : H_{{\Cubical}_{c}}^{p}(U_{1}) \oplus H_{{\Cubical}_{c}}^{p}(U_{2}) \to H_{{\Cubical}_{c}}^{p}(X)$ by $\phi_{*}([\omega]) = i_{1*}[\omega]{\oplus}i_{2*}[\omega]$ and $\psi_{*}([\eta_{1}]{\oplus}[\eta_{2}]) = j_{1*}[\eta_{1}]-j_{2*}[\eta_{2}]$.
Then we obtain the following long exact sequence. 
\par\vskip1.0ex\noindent
\hfil$\begin{array}{l}
H_{{\Cubical}_{c}}^{0}(U_{0}) \rightarrow \cdots \rightarrow H_{{\Cubical}_{c}}^{p}(U_{0}) 
\xrightarrow{\phi_{*}} H_{{\Cubical}_{c}}^{p}(U_{1}){\oplus}H_{{\Cubical}_{c}}^{p}(U_{2}) 
\xrightarrow{\psi_{*}} H_{{\Cubical}_{c}}^{p}(X)
\\[1.5ex]\qquad\qquad
\xrightarrow{\diff{}_{*}} H_{{\Cubical}_{c}}^{p+1}(U_{0}) 
\xrightarrow{\phi_{*}} H_{{\Cubical}_{c}}^{p+1}(U_{1}){\oplus}H_{{\Cubical}_{c}}^{p+1}(U_{2}) 
\xrightarrow{\psi_{*}} H_{{\Cubical}_{c}}^{p+1}(X) \rightarrow \cdots.
\end{array}$\hfil
\end{thm}\vspace{1.5ex}
\begin{proof}
For any closed subsets $G'_{t} \supset G_{t}$ in $U_{t}$, there is a following short exact sequence.
\par\vskip.5ex\noindent
\hfil$
0 \longrightarrow \Omega_{\Cubical_{c}}^{p}(\mathcal{U}_{0};G'_{0}) 
\overset{\phi_{\natural}}\longrightarrow \Omega_{\Cubical_{c}}^{p}(\mathcal{U}_{1};G'_{1}) \oplus \Omega_{\Cubical_{c}}^{p}(\mathcal{U}_{2};G'_{2}) 
\overset{\psi_{\natural}}\longrightarrow \Omega_{\Cubical_{c}}^{p}(\mathcal{U}_{3};X) \longrightarrow 0, 
$\hfil
\par\vskip.5ex\noindent
where $G'_{0}=G'_{1} \cap G'_{2}$, \,$\mathcal{U}_{0}=\{U_{0}\}$, \,$\mathcal{U}_{t}=\{U_{0},U_{t}\!\smallsetminus\!G'_{3-t}\}$, $t=1, 2$ and \,$\mathcal{U}_{3}=\{U_{0},U_{1}\!\smallsetminus\!G'_{2},U_{2}\!\smallsetminus\!G'_{1}\}$, which are open coverings of $U_{0}$, $U_{t}$ and $X$, respectively.
\begin{description}
\item[(exactness at $\Omega_{\Cubical_{c}}^{p}(\mathcal{U}_{0};G'_{0})$)]
Assume $\phi_{\natural}(\omega)=0$.
Then $i_{1\sharp}(\omega) = i_{2\sharp}(\omega) = 0$.
Since $i_{1\sharp}(\omega)$ is an extension of $\omega$, we obtain $\omega=0$.
Thus $\phi_{\natural}$ is a monomorphism.
\vspace{.5ex}
\item[(exactness at $\Omega_{\Cubical_{c}}^{p}(\mathcal{U}_{1};G'_{1}) \oplus \Omega_{\Cubical_{c}}^{p}(\mathcal{U}_{2};G'_{2})$)]
Assume $\psi_{\natural}(\eta^{(1)}{\oplus}\eta^{(2)})=0$.
Then we have $j_{1\sharp}(\eta^{(1)})=j_{2\sharp}(\eta^{(2)})$.
For any plot $P : \simp^{n} \to X$, we obtain 
$j_{1\sharp}(\eta^{(1)})_{\underline{n}}(P)=j_{2\sharp}(\eta^{(2)})_{\underline{n}}(P)$.
So, for any plot $Q : \simp^{m} \to U_{0}$, $\eta^{(1)}_{B}(i_{1}{\circ}Q) = j_{1}^{\sharp}\eta^{(1)}_{\underline{m}}(j_{1}{\circ}i_{1}{\circ}Q) = j_{2}^{\sharp}\eta^{(2)}_{\underline{m}}(j_{2}{\circ}i_{2}{\circ}Q) = \eta^{(2)}_{\underline{m}}(i_{2}{\circ}Q)$.
Then, we define $\eta^{(0)} \in \Omega^{p}_{\Cubical}(U_{0})$ by $\eta^{(0)}_{\underline{m}}(Q)$ $=$ $\eta^{(1)}_{\underline{m}}(i_{1}{\circ}Q)$ $=$ $\eta^{(2)}_{\underline{m}}(i_{2}{\circ}Q)$.
On the other hand, $K_{j_{t\sharp}\eta^{(t)}} = K_{\eta^{(t)}}$ by definition, and hence we obtain 
\par\vskip1ex\noindent\hfil
$\Supp{\eta^{(0)}_{\underline{m}}(Q)} = \Supp{\eta^{(1)}_{\underline{m}}(i_{1}{\circ}Q)} = \Supp{\eta^{(2)}_{\underline{m}}(i_{2}{\circ}Q)} \subset Q^{-1}(K_{\eta^{(1)}} \cap K_{\eta^{(2)}})$.
\hfil\par\vskip1ex\noindent
Then we have $\eta^{(0)} \in \Omega_{\Cubical_{c}}^{p}(\mathcal{U}_{0})$, for $K_{\eta^{(0)}}=K_{\eta^{(1)}} \cap K_{\eta^{(2)}}$ is compact in $U_{0}$, which satisfies $\phi_{\natural}(\eta^{(0)})=(\eta^{(1)},\eta^{(2)})$.
Thus $(\eta^{(1)},\eta^{(2)})$ is in the image of $\phi_{\natural}$.
The other direction is clear by definition and it implies the  exactness at $\Omega_{\Cubical_{c}}^{p}(\mathcal{U}_{1};G'_{1}) \oplus \Omega_{\Cubical_{c}}^{p}(\mathcal{U}_{2};G'_{2})$.
\vspace{.5ex}
\item[(exactness at $\Omega_{\Cubical_{c}}^{p}(\mathcal{U}_{3};X)$)]
Assume $\kappa \in \Omega_{\Cubical_{c}}^{p}(\mathcal{U}_{3};X)$.
For any plot $P_{t} : \simp^{n_{t}} \to U_{t}$, we define $\kappa^{(t)}_{\underline{n_{t}}}(P_{t})(\mathbold{x})$ by $(-1)^{t-1}\rho^{(t)}_{\underline{n_{t}}}(P_{t})(\mathbold{x}){\cdot}\kappa_{\underline{n_{t}}}(j_{t}{\circ}P_{t})(\mathbold{x})$ if $\mathbold{x} \in P_{t}^{-1}(U_{0})$ and by $0$ if $\mathbold{x} \not\in \Supp{\rho_{\underline{n_{t}}}^{(t)}}(P_{t})$.
Then $\kappa^{(t)}$ is a differential $p$-form on $U_{t}$ and $\kappa^{(t)} \in \Omega_{{\Cubical}_{c}}^{p}(U_{t})$ for $K_{\kappa^{(t)}} = K_{\kappa} \cap G_{t} \subset G'_{t}$ is compact in $U_{t}$.
Then we have $\psi_{\natural}(\kappa^{(1)}{\oplus}\kappa^{(2)}) = \kappa$, and hence $\kappa$ is in the image of $\psi_{\natural}$.
Thus $\psi_{\natural}$ is an epimorphism.\vspace{-.5ex}
\end{description}
Since $\phi_{\natural}$ and $\psi_{\natural}$ are clearly cochain maps, we obtain the following long exact sequence.
\par\vskip.5ex\noindent\hfil
$\begin{array}{l}
H^{0}_{\Cubical_{c}}(\mathcal{U}_{0};G'_{0}) \rightarrow 
\cdots \rightarrow H^{p}_{\Cubical_{c}}(\mathcal{U}_{0};G'_{0})
\xrightarrow{\overline{\phi}_{*}} H^{p}_{\Cubical_{c}}(\mathcal{U}_{1};G'_{1}){\oplus}H^{p}_{\Cubical_{c}}(\mathcal{U}_{2};G'_{2})
\xrightarrow{\overline{\psi}_{*}} H_{{\Cubical}_{\Cubical_{c}}}^{p}(\mathcal{U}_{3})
\\[1ex]\qquad\qquad
\xrightarrow{\overline{\diff{}}_{*}} H^{p+1}_{\Cubical_{c}}(\mathcal{U}_{0};G'_{0}) 
\xrightarrow{\overline{\phi}_{*}} H^{p+1}_{\Cubical_{c}}(\mathcal{U}_{1};G'_{1}){\oplus}H^{p+1}_{\Cubical_{c}}(\mathcal{U}_{2};G'_{2}) 
\xrightarrow{\overline{\psi}_{*}} H_{{\Cubical}_{\Cubical_{c}}}^{p+1}(\mathcal{U}_{3}) \rightarrow \cdots.
\end{array}$\hfil
\par\vskip.5ex\noindent
So we can define connecting homomorphism $\diff{}_{*} : H_{{\Cubical}_{c}}^{p}(X) \overset{\res^{*}}\cong H_{{\Cubical}_{c}}^{p}(\mathcal{U}_{3}) \xrightarrow{\overline{\diff{}_{*}}} H^{p+1}_{c}(\mathcal{U}_{0};G'_{0}) \to H_{{\Cubical}_{c}}^{p+1}(U_{0})$ where the latter map is induced from the natural inclusion $\Omega^{p+1}_{{\Cubical}_{c}}(\mathcal{U}_{0};G'_{0}) \subset \Omega_{{\Cubical}_{c}}^{p+1}(\mathcal{U}_{0}) = \Omega_{{\Cubical}_{c}}^{p+1}(U_{0})$, which fits in with the following commutative ladder.
\par\vskip.5ex\noindent\hfil
$\divide\dgARROWLENGTH by 2
\begin{diagram}
\node{H^{p}_{\Cubical_{c}}(\mathcal{U}_{0};G'_{0})}
\arrow{s}
\arrow{r,t}{\overline{\phi_{*}}}
\node{H^{p}_{\Cubical_{c}}(\mathcal{U}_{1};G'_{1}){\oplus}H^{p}_{\Cubical_{c}}(\mathcal{U}_{2};G'_{2})} 
\arrow{s}
\arrow{r,t}{\overline{\psi_{*}}} 
\node{H_{{\Cubical}_{c}}^{p}(\mathcal{U}_{3})}
\arrow{s,=}
\arrow{r,t}{\overline{\diff{}_{*}}}
\node{H^{p+1}_{\Cubical_{c}}(\mathcal{U}_{0};G'_{0})}
\arrow{s}
\\
\node{H^{p}_{{\Cubical}_{c}}(\mathcal{U}_{0})}
\node{H^{p}_{{\Cubical}_{c}}(\mathcal{U}_{1}){\oplus}H^{p}_{{\Cubical}_{c}}(\mathcal{U}_{2})} 
\node{H_{{\Cubical}_{c}}^{p}(\mathcal{U}_{3})}
\node{H^{p+1}_{{\Cubical}_{c}}(\mathcal{U}_{0})}
\\
\node{H_{{\Cubical}_{c}}^{p}(U_{0})}
\arrow{n,=}
\arrow{r,t}{\phi_{*}}
\node{H_{{\Cubical}_{c}}^{p}(U_{1}){\oplus}H_{{\Cubical}_{c}}^{p}(U_{2})} 
\arrow{n,lr}{\cong}{\res^{*}{\oplus}\res^{*}}
\arrow{r,t}{\psi_{*}} 
\node{H_{{\Cubical}_{c}}^{p}(X)}
\arrow{n,lr}{\cong}{\res^{*}}
\arrow{r,t}{\diff{}_{*}} 
\node{H_{{\Cubical}_{c}}^{p+1}(U_{0})}
\arrow{n,=}
\end{diagram}$\hfil
\par\vskip.5ex\noindent
Using these diagrams, we show the desired exactness as follows.
\begin{description}
\item[(exactness at $H_{{\Cubical}_{c}}^{p}(U_{0})$)]
Assume $\phi_{*}([\omega])=0$.
Let $G'_{t}=G_{t} \cup K_{\omega}$, $t=0, 1, 2$.
Then $[\omega] \in H^{p}_{\Cubical_{c}}(\mathcal{U}_{0};G'_{0})$ satisfying $\overline{\phi}_{*}([\omega])$ is zero in $H^{p}_{{\Cubical}_{c}}(\mathcal{U}_{1}){\oplus}H^{p}_{{\Cubical}_{c}}(\mathcal{U}_{2})$.
Hence there is $\sigma^{(1)}{\oplus}\sigma^{(2)} \in \Omega^{p}_{{\Cubical}_{c}}(\mathcal{U}_{1}){\oplus}\Omega^{p}_{{\Cubical}_{c}}(\mathcal{U}_{2})$ such that $\diff{\sigma^{(1)}}{\oplus}\diff{\sigma^{(2)}}=\phi_{\natural}(\omega)$.
Then we may expand $G'_{t}$ as $G'_{t}=G_{t} \cup K_{\omega} \cup K_{\sigma^{(t)}}$, $t=1, 2$ and $G'_{0}=G'_{1} \cap G'_{2}$, so that we obtain $\overline{\phi}_{*}([\omega])=0$, and hence $[\omega] \in \Im{\overline{\diff{}}_{*}}$ in $H^{p+1}_{\Cubical_{c}}(\mathcal{U}_{0};G'_{0})$.
\vspace{1ex}
Thus $[\omega]$ is in the image of $\diff{}_{*}$.
\item[(exactness at $H_{{\Cubical}_{c}}^{p}(U_{1}) \oplus H_{{\Cubical}_{c}}^{p}(U_{2})$)]
Assume $\psi_{*}([\eta^{(1)}]{\oplus}[\eta^{(2)}])=0$.
Let $G'_{t} = G_{t} \cup K_{\eta^{(t)}}$, $t=1, 2$ and $G'_{0}=G'_{1} \cap G'_{2}$, so that $[\eta^{(1)}]{\oplus}[\eta^{(2)}] \in H^{p}_{\Cubical_{c}}(\mathcal{U}_{1};G'_{1}){\oplus}H^{p}_{\Cubical_{c}}(\mathcal{U}_{2};G'_{2})$ and $\overline{\psi}_{*}([\eta^{(1)}]{\oplus}[\eta^{(2)}])=0$ in $H_{{\Cubical}_{c}}^{p}(\mathcal{U}_{3}) \cong H_{{\Cubical}_{c}}^{p}(X)$.
Then we obtain $[\eta^{(1)}]{\oplus}[\eta^{(2)}] \in \Im{\overline{\phi}_{*}}$ in $H_{\Cubical_{c}}^{p}(\mathcal{U}_{1};G'_{1}) \oplus H_{\Cubical_{c}}^{p}(\mathcal{U}_{2};G'_{2})$, and hence $[\eta^{(1)}]{\oplus}[\eta^{(2)}]$ is in the image of $\phi_{*}$.
\vspace{1ex}
\item[(exactness at $H_{{\Cubical}_{c}}^{p}(X)$)]
Assume $\diff{}_{*}([\kappa]) = 0$. 
Then there is $\sigma \in \Omega^{p}_{{\Cubical}_{c}}(U_{0})$ such that $\diff{}_{\natural}(\kappa) = \diff{\sigma}$ in $\Omega^{p+1}_{{\Cubical}_{c}}(U_{0})$.
Let $G'_{t}=G_{t} \cup K_{\sigma}$, $t=0, 1, 2$.
Then we may assume $\sigma \in \Omega^{p}_{\Cubical_{c}}(\mathcal{U}_{0};G'_{0})$ satisfying $\diff{}_{\natural}(\kappa)=\diff{\sigma}$ in $\Omega^{p+1}_{\Cubical_{c}}(\mathcal{U}_{0};G'_{0})$, and hence $[\kappa] \in \Im{\overline{\psi}_{*}}$ in $H_{{\Cubical}_{c}}^{p}(\mathcal{U}_{3})$.
Thus $[\kappa]$ is in the image of $\psi_{*}$.
\end{description}
The other directions are clear by definition, and it completes the proof of the theorem.
\end{proof}

Let $\topological$ be the category of topological spaces and continuous maps.
Then there are natural embeddings $\topological \hookrightarrow \differentiable$ and $\topological \hookrightarrow \diffeological$. 

Let $X=(X,\{X^{(n)}\,;\,n\!\geq\!-1\})$ be a topological CW complex embedded in the category $\diffeological$ or $\differentiable$ with the set of $n$-balls $\{B^{n}_{j}\}$ indexed by $j \!\in\!J_{n}$. 
Then we have open sets $U=X^{(n)} \smallsetminus X^{(n-1)}$ and $V=X^{(n)}\smallsetminus (\underset{\scriptsize j \in J_{n}}{\cup}\{\mathbold{0}_{j}\})$ in $X^{(n)}$, where $\mathbold{0}_{j} \in B^{n}_{j}$ denotes the element corresponding to $\mathbold{0} \in B^{n}=\{\mathbold{x} \in \real^{n}\,;\,\Vert\mathbold{x}\Vert \leq 1\}$ the origin of $\real^{n}$.

A ball $B^{n}_{j} = B^{n}$ (if we disregard the indexing) has a nice open covering given by $\{\Int{B^{n}_{j}},B^{n}_{j}{\smallsetminus}\{\mathbold{0}\}\}$ with a partition of unity $\{\rho_{1}^{(j)},\rho_{2}^{(j)}\}$ as follows: $\rho_{1}^{(j)}=1{-}\rho_{2}^{(j)}$ and $\rho_{2}^{(j)}(\mathbold{x})=\lambda(\Vert{\mathbold{x}}\Vert)$ for small $a\!>\!0$.
Thus $\mathcal{U} = \{U,V\}$ is a nice open covering of $X^{(n)}$ with a normal partition of unity $\{\rho^{U},\rho^{V}\}$ in which $\rho^{U}$ is a zero-extension of $\rho_{1}^{(j)}$'s on the union of balls and $\rho^{V}=1{-}\rho^{U}$. 
Then $U$ is smoothly homotopy equivalent to discrete points each of which is $\mathbold{0}_{j} \in B^{n}_{j}$ for some $j \!\in\! J_{n}$ and $V$ is smoothly homotopy equivalent to $X^{(n-1)}$.
By comparing Mayer-Vietoris sequences associated to $\mathcal{U}$ in Theorem \ref{thm:nice-open-covering-1} with that in Theorem \ref{thm:cubical-open-covering} for $X = X^{(n)}$, we obtain the following result using Remark \ref{rem:manifold-C} together with so-called five lemma, by using standard homological methods inductively on $n$.

\begin{thm}\label{thm:deRham-top}
For a CW complex $X$, there are natural isomorphisms
\par\vskip.5ex\noindent\hfil$\begin{array}{l}
H_{\mathcal{D}}^{q}(X) \cong 
H_{\mathcal{C}}^{q}(X) \cong H_{\Cubical}^{q}(X) \cong H^{q}(X,\real) \cong \Hom(H_{q}(X),\real),
\end{array}$
\hfil\par\vskip0ex\noindent
for any $q \geq 0$, and hence we have $H_{\mathcal{D}}^{1}(X) \cong H_{\mathcal{C}}^{1}(X) \cong H_{\Cubical}^{1}(X) \overset{\rho}\cong \Hom(\pi_{1}(X),\real)$.
\end{thm}
\begin{conj}\label{cor:deRham-compact-support-top}
For a CW complex $X$, there are natural isomorphisms
\par\vskip.5ex\noindent\hfil
$\begin{array}{l}
H_{\mathcal{D}_{c}}^{q}(X) \cong 
H_{\mathcal{C}_{c}}^{q}(X) \cong H_{\Cubical_{c}}^{q}(X),
\end{array}$ for any $q \geq 0$.
\end{conj}

It would be possible to determine $H_{\Cubical}^{\ast}(X)$ and $H_{\Cubical_{c}}^{\ast}(X)$ by using standard methods in algebraic topology even if $X$ is not a topological CW complex, while we do not know how to determine $H_{\mathcal{D}}^{\ast}(X)$, $H_{\mathcal{C}}^{\ast}(X)$, $H_{\mathcal{D}_{c}}^{\ast}(X)$ nor $H_{\mathcal{C}_{c}}^{\ast}(X)$, if we do not find out any appropriate nice open covering (with a normal partition of unity) on $X$.\vspace{-1ex}

\appendix

\section{Smooth CW complex}

A smooth CW complex $X=(X,\{X^{(n)}\,;\,n\!\geq\!-1\})$ is a differentiable or diffeological space built up from $X^{(-1)} = \emptyset$ by inductively attaching $n$-balls $\{B^{n}_{j}\,;\,j \!\in\! J_{n}\}$ by $C^{\infty}$ maps from their boundary spheres $\{S^{n-1}_{j}\,;\,j \!\in\! J_{n}\}$ to $n{-}1$-skeleton $X^{(n-1)}$ to obtain $n$-skeleton $X^{(n)}$, $n \!\geq\! 0$, where the differentiable structures of balls and boundary spheres are given by their manifold structures.
Thus a plot in $X^{(n)}$ is a map $P : A \to X$ with an open covering $\{A_{\alpha}\,;\,\alpha\!\in\!\Lambda\}$ of $A$ such that $P(A_{\alpha})$ is in $X^{(n-1)}$ or $B^{n}_{j}$ for some $j \!\in\! J_{n}$ and $P\vert_{A_{\alpha}}$ is a plot of $X^{(n-1)}$ or $B^{n}_{j}$, respectively.
Then $X$ is the colimit of $\{X^{(n)}\}$ in $\differentiable$ or $\diffeological$.

Let $X=(X,\{X^{(n)}\})$ be a smooth CW complex in either $\differentiable$ or $\diffeological$ with the set of $n$-balls $\{B^{n}_{j}\,;\,j \!\in\!J_{n}\}$.
Then for any plot $P : A \to X^{(n)}$, there is an open covering $\{A_{\alpha}\}$ of $A$, such that $P(A_{\alpha})$ is in either $X^{(n-1)}$ or $B^{n}_{j}$ for some $j \!\in\! J_{n}$ and $P_{\alpha}=P\vert_{A_{\alpha}}$ is a plot of $X^{(n-1)}$ or $B^{n}_{j}$,\vspace{.5ex} respectively.
Let $U=X^{(n)} \smallsetminus X^{(n-1)}$ and $V=X^{(n)}\smallsetminus (\underset{\scriptsize j \in J_{n}}{\cup}\{\mathbold{0}_{j}\})$, where $\mathbold{0}_{j} \in B^{n}_{j}$ denotes the element corresponding to $\mathbold{0} \in B^{n}$.
\begin{description}
\item[\quad Case ($\Im{P_{\alpha}} \subset X^{(n-1)}$)]
$P_{\alpha}^{-1}(U) = \emptyset$,\vspace{.5ex} and $P_{\alpha}^{-1}(V) = A_{\alpha}$.
\vskip.5ex
\item[\quad Case ($\Im{P_{\alpha}} \subset B^{n}_{j}$)]
$P_{\alpha}^{-1}(U) = P_{\alpha}^{-1}(\Int{B^{n}_{j}})$, and $P_{\alpha}^{-1}(V) = P_{\alpha}^{-1}(B^{n}_{j} \smallsetminus \{\mathbold{0}_{j}\})$.
\end{description}
\vspace{.5ex}
In each case, $P_{\alpha}^{-1}(U)$ and $P_{\alpha}^{-1}(V)$ are open in $A_{\alpha}$ and hence in $A$, which implies that $P^{-1}(U)$ and $P^{-1}(V)$ are open in $A$ for any plot $P$.
Thus $U$ and $V$ are open sets in $X^{(n)}$.

Similarly to the case when $X$ is a topological CW complex, 
$\mathcal{U} = \{U,V\}$ is a nice open covering of $X^{(n)}$ with a normal partition of unity $\{\rho^{U},\rho^{V}\}$, since $\lambda$ is a smooth function. 
Then, similar arguments for a topological CW complex lead us to the following result. 

\begin{thm}\label{thm:deRham-smooth}
For a smooth CW complex $X$, there are natural isomorphisms
\par\vskip.5ex\noindent\hfil$\begin{array}{l}
H_{\mathcal{D}}^{q}(X) \cong H_{\mathcal{C}}^{q}(X) \cong H_{\Cubical}^{q}(X) \cong H^{q}(X,\real) \cong \Hom(H_{q}(X),\real),
\end{array}$\hfil\par\vskip0ex\noindent
for any $q \geq 0$, and hence we have $H_{\mathcal{D}}^{1}(X) \cong H_{\mathcal{C}}^{1}(X) \cong H_{\Cubical}^{1}(X) \overset{\rho}\cong \Hom(\pi_{1}(X),\real)$.
\end{thm}
\begin{conj}\label{conj:deRham-compact-support-smooth}
For a smooth CW complex $X$, there are natural isomorphisms
\par\vskip.5ex\noindent\hfil$\begin{array}{l}
H_{\mathcal{D}_{c}}^{q}(X) \cong 
H_{\mathcal{C}_{c}}^{q}(X) \cong H_{\Cubical_{c}}^{q}(X),
\end{array}$
for any \ $q \geq 0$.
\end{conj}

%
%
\subsection*{Acknowledgements}
This research was supported by Grant-in-Aid for Scientific Research (B) \#22340014, (A) \#23244008 and for Exploratory Research \#24654013 from Japan Society for the Promotion of Science.

%
%


\begin{thebibliography}{99}
%
\bibitem{BH:11}
J.C.~Baes and A.E.~Hoffnung, 
{\em Convenient categories of smooth spaces}, 
Trans. Amer. Math. Soc., \textbf{363} (2011), 5789--5825.
%
\bibitem{Chen:73}
K.~T.~Chen, 
{\em Iterated integrals of differential forms and loop space homology}, 
Ann. of Math. (2) \textbf{97} (1973), 217--246.
%
\bibitem{Chen:75}
K.~T.~Chen, 
{\em Iterated Integrals, Fundamental Groups and Covering Spaces}, 
Trans. Amer. Math. Soc., \textbf{206} (1975), 83--98.
%
\bibitem{Chen:77}
K.~T.~Chen, {\em Iterated path integrals}, Bull. Amer.Math. Soc., \textbf{83}, (1977), 831--879.
%
\bibitem{Chen:86}
K.~T.~Chen, 
{\em On differentiable spaces}, 
Categories in Continuum Physics, Lecture Notes in Math., \textbf{1174}, Springer, Berlin, 1986, 38--42.
%
\bibitem{Haraguchi:14}
T.~Haraguchi,
{\em Long exact sequences for de Rham cohomology of diffeological spaces},
Kyushu J. Math. \textbf{68} (2014), 333--345.
%
\bibitem{Iglesias:13}
P.~Iglesias-Zemmour, ``Diffeology'', Mathematical Surveys and Monographs, \textbf{185}, Amer. Math. Soc., New York 2013.
%
\bibitem{Izumida:14}
N.~Izumida, {\em de Rham theory in Diffeology}, Master Thesis, Kyushu University, 2014.
%
\bibitem{KM:14}
A.~Kriegl and P.~W.~Michor, ``The convenient setting of global analysis'', Mathematical Surveys and Monographs, \textbf{53}, Amer. Math. Soc., New York 1996.
%
\bibitem{Souriau:80}
J.~M.~Souriau, 
{\em Groupes differentiels}, 
in ``Differential Geometrical Methods in Mathematical Physics'' (Proc. Conf. Aix-en-Provence/Salamanca, 1979), Lecture Notes in Math., \textbf{836}, Springer, Berlin, 1980, 91--128.
\bibitem{Stacey:11}
A.~Stacey,
{\em Comparative smootheology},
Theory and Applications of Categories \textbf{25} (2011), 64--117.
%
\bibitem{Wu:12}
E.~Wu,
{\em A Homotopy Theory for Diffeological Spaces}, Thesis, University of Western Ontario, 2012.
\\
\end{thebibliography}
\end{document}